\documentclass[secthm,seceqn,amsthm,ussrhead,12pt]{amsart}
\usepackage[utf8]{inputenc}
\usepackage[english]{babel}

\usepackage{amssymb,amsmath,amsthm,amsfonts,xcolor,enumerate,hyperref,comment,longtable,cleveref}
\usepackage{soul}
\usepackage{times}
\usepackage{cite}
\usepackage{pdflscape}
\usepackage{ulem}
\usepackage[mathcal]{euscript}
\usepackage{tikz}
\usepackage{hyperref}
\usepackage{cancel}
\usepackage{stmaryrd}

\usetikzlibrary{arrows}

\newtheorem{theorem}{Theorem}
\newtheorem{lemma}[theorem]{Lemma}

\newtheorem{remark}[theorem]{Remark}

\newtheorem{definition}[theorem]{Definition}

\newtheorem*{theoremA}{Theorem A}
\newtheorem*{corollaryA}{Corollary}
\newtheorem*{theoremB}{Theorem B}
\newtheorem*{theoremC}{Theorem C}
\newtheorem*{theoremD}{Theorem D}

\newtheorem*{theoremE}{Theorem E}
\newtheorem*{theoremF}{Theorem F}
\newenvironment{Proof}[1][Proof.]{\begin{trivlist}
\item[\hskip \labelsep {\bfseries #1}]}{\flushright
$\Box$\end{trivlist}}

\usepackage{stmaryrd}
\usepackage{xcolor}

\begin{document}

\noindent{\Large
The algebraic and geometric classification of \\ 
nilpotent weakly associative  and symmetric Leibniz algebras}

   \medskip

  \medskip
  
   {\bf  
   Mar\'ia Alejandra Alvarez 
  $\&$
    Ivan Kaygorodov
    \\

    \medskip
}

{\tiny




 \medskip

\medskip

   E-mail addresses: 
   
\smallskip   
    Mar\'ia Alejandra Alvarez (maria.alvarez@uantof.cl)

\smallskip
    Ivan Kaygorodov (kaygorodov.ivan@gmail.com)

}

\medskip
 
\ 

\noindent{\bf Abstract}:
{\it This paper is devoted to the complete algebraic and geometric classification of complex $4$-dimensional nilpotent weakly associative, 
complex $4$-dimensional symmetric Leibniz  algebras,
and  complex $5$-dimensional nilpotent  symmetric Leibniz algebras.
In particular, we proved that the variety of complex $4$-dimensional symmetric Leibniz algebras has no Vergne--Grunewald--O'Halloran Property
(there is an irreducible component formed by only nilpotent algebras), but on the other hand, it has Vergne Property
(there are no rigid nilpotent algebras).
 }
 
\medskip

\medskip

\medskip


 \medskip

\section*{Introduction}
The algebraic classification (up to isomorphism) of algebras of dimension $n$ from a certain variety
defined by a certain family of polynomial identities is a classic problem in the theory of non-associative algebras.
There are many results related to the algebraic classification of small-dimensional algebras in the varieties of
Jordan, Lie, Leibniz, Zinbiel and many other algebras \cite{ack,  kkl20,          degr3, usefi1,  degr1, degr2,    ha16,    kkp20,   kv16}.
 Geometric properties of a variety of algebras defined by a family of polynomial identities have been an object of study since 1970's (see, \cite{wolf2, wolf1,     chouhy,     BC99, aleis, aleis2,   gabriel,   ckls, cibils,  shaf, GRH, GRH2, ale3,     ikv17,   kppv,    kv16, S90}). 
 Gabriel described the irreducible components of the variety of $4$-dimensional unital associative algebras~\cite{gabriel}.  
 Cibils considered rigid associative algebras with $2$-step nilpotent radical \cite{cibils}.
 Burde and Steinhoff  constructed the graphs of degenerations for the varieties of    $3$-dimensional and $4$-dimensional Lie algebras~\cite{BC99}. 
 Grunewald and O'Halloran  calculated the degenerations for the variety of $5$-dimensional nilpotent Lie algebras~\cite{GRH}. 
Chouhy  proved that  in the case of finite-dimensional associative algebras,
 the $N$-Koszul property is preserved under the degeneration relation~\cite{chouhy}.
Degenerations have also been used to study a level of complexity of an algebra~\cite{wolf1,wolf2}.
 The study of degenerations of algebras is very rich and closely related to deformation theory, in the sense of Gerstenhaber \cite{ger63}.

 \newpage 

In the present paper, we give the algebraic  and geometric classification of
complex $4$-dimensional  nilpotent weakly associative  and symmetric Leibniz algebras, 
 and complex $5$-dimensional nilpotent symmetric Leibniz algebras.
An anticommutative weakly associative algebra is a Lie algebra and the variety of weakly associative algebras is a proper subvariety of the variety of flexible algebras defined by the following identity $(xy)x=x(yx).$
In dimension $2$, these two varieties coincide. Therefore, the algebraic and geometric classification of $2$-dimensional weakly associative algebras follows from \cite[Section $7.1$]{kv16}.
The variety of weakly associative algebras contains commutative algebras, associative algebras, Lie algebras and symmetric Leibniz algebras as subvarieties \cite{remm1, remm2}. 
At the same time, 
the variety of symmetric Leibniz algebras   is in intersection of right Leibniz and left Leibniz algebras \cite{pirashvili, pirashvilinew, mason, bsaid, benedikt, elisabete} 
(for superalgebra case see,  \cite{saidsuper})
and it  plays an important role in one-sided Leibniz algebras 
(more, about Leibniz and symmetric Leibniz algebras see, \cite{Feldvoss}).
So, each quadratic (i.e. endowed with a bilinear, symmetric and non-degenerate associative form) Leibniz algebra is a symmetric Leibniz algebra \cite{saidqua}.
Symmetric Leibniz algebras are related to  Lie racks \cite{racks}. 
Every symmetric Leibniz algebra is flexible, power-associative and nil-algebra with nilindex $3$ \cite{Feldvoss}.
Thanks to  \cite[Theorem 2.2]{saidsuper} each symmetric Leibniz algebra is a central extension of a suitable Lie algebra. 
It satisfies the following identities:
\begin{center} 
$x(yz)=(xy)z+y(xz)$ \ and \  $(xy)z=(xz)y+x(yz).$\end{center}
On the other hand, weakly associative algebras are related to non-associative Poisson algebras. Namely, for a weakly associative algebra the algebraic system with the same underline space endowed with the commutator multiplication and the anticommutator (Jordan) multiplication gives a non-associative Poisson algebra
\cite{remm1, remm2}.
A symmetric Leibniz algebra under  commutator and anticommutator multiplications gives a Poisson algebra \cite{said2}. 
The variety of weakly associative algebras is defined by the following identity:
\begin{center}
$(xy)z-x(yz)+(yz)x-y(zx) = (yx)z- y(xz).$
\end{center} 
The definition of  the weakly associative algebras can be re-written in the following way: {\it for each element $x$  it is true that $L_x-R_x$  gives a derivation of the algebra}, where  $L_x$ and $R_x$ are, respectively, operators of left and right multiplications of element $x.$ 
Recently, another type of non-associative algebras, called $\mathfrak{CD}$-algebras, has been introduced by a similar way: {\it for each elements $x,y$  it is true that the commutator of operators of one sided multiplication gives a derivation} (see, \cite{  ack}).

Our method for classifying nilpotent weakly associative and symmetric Leibniz algebras is based on the calculation of central extensions of nilpotent algebras of smaller dimensions from the same variety.
The algebraic study of central extensions of   algebras has been an important topic for years \cite{  klp20,hac16,  ss78}.
First, Skjelbred and Sund used central extensions of Lie algebras to obtain a classification of nilpotent Lie algebras  \cite{ss78}.
Note that the Skjelbred-Sund method of central extensions is an important tool in the classification of nilpotent algebras.
Using the same method,  
 $4$-dimensional nilpotent 
(associative \cite{degr1}, 
 terminal  \cite{kkp20}, 
 commutative  \cite{fkkv},
 right commutative \cite{akks21}) algebras,
  $5$-dimensional nilpotent 
(Jordan \cite{ha16},
 restricted Lie  \cite{usefi1}) algebras,
 $6$-dimensional nilpotent 
(Lie  \cite{degr3,degr2}, 
 binary Lie  \cite{ack},
 anticommutative  \cite{kkl20}) algebras,
 $8$-dimensional   dual Mock-Lie algebras \cite{ckls},
and some others have been described. Our main results related to the algebraic classification of cited varieties are summarized below.

\begin{theoremA}
Up to isomorphism, there are only five complex $4$-dimensional nilpotent (non-$2$-step nilpotent) symmetric Leibniz algebras, described explicitly in  section \ref{secteoA}.
Up to isomorphism, there are infinitely many isomorphism classes of  
complex  $4$-dimensional nilpotent (non-commutative, non-symmetric Leibniz) weakly associative  algebras, 
described explicitly  in  section \ref{secteoA} in terms of $2$ one-parameter families and $9$ additional isomorphism classes.
\end{theoremA}

\begin{theoremB}
Up to isomorphism, there are infinitely many isomorphism classes of  
complex  $4$-dimensional  (non-nilpotent, non-Lie) symmetric Leibniz  algebras, 
described explicitly  in  section \ref{secteoA} in terms of 
$3$ one-parameter families and $6$ additional isomorphism classes.
\end{theoremB}

\begin{theoremC}
Up to isomorphism, there are infinitely many isomorphism classes of  
complex  $5$-dimensional  (non-$2$-step nilpotent) nilpotent symmetric Leibniz  algebras, 
described explicitly  in  section \ref{secteoB} in terms of 
$1$ two-parameter family,
$6$ one-parameter families and 
$42$ additional isomorphism classes.
\end{theoremC}

 The degenerations between the (finite-dimensional) algebras from a certain variety $\mathfrak{V}$ defined by a set of identities have been actively studied in the past decade.
The description of all degenerations allows one to find the so-called rigid algebras and families of algebras, i.e. those whose orbit closures under the action of the general linear group form irreducible components of $\mathfrak{V}$
(with respect to the Zariski topology). 
We list here some works in which the rigid algebras of the varieties of
all $4$-dimensional Leibniz algebras \cite{ikv17},
all $4$-dimensional nilpotent terminal algebras \cite{kkp20},
all $4$-dimensional nilpotent commutative algebras \cite{fkkv},
all $6$-dimensional nilpotent binary Lie algebras \cite{ack},
all $6$-dimensional nilpotent anticommutative algebras \cite{kkl20},
all $8$-dimensional dual Mock Lie algebras \cite{ckls}
have been found.
A full description of degenerations has been obtained  
for $2$-dimensional algebras in \cite{kv16}, 
for $4$-dimensional Lie algebras in \cite{BC99},
for $4$-dimensional Zinbiel and  $4$-dimensional nilpotent Leibniz algebras in \cite{kppv},
for $6$-dimensional nilpotent Lie algebras in \cite{S90,GRH},  
for $8$-dimensional $2$-step nilpotent anticommutative algebras \cite{ale3},
and for $(n+1)$-dimensional $n$-Lie algebras.
Our main results related to the geometric classification of cited varieties are summarized below. 
\begin{theoremD}
The variety of complex $4$-dimensional nilpotent symmetric Leibniz algebras has 
dimension {\it 11 }  and it has 
three  irreducible components (in particular, there is only one rigid algebra in this variety). The variety of $4$-dimensional nilpotent weakly associative   algebras  has dimension $16$  and it has
  three irreducible components 
(in particular, there is only one rigid algebra in this variety).

\end{theoremD}

\begin{theoremE}
The variety of  complex $4$-dimensional  symmetric Leibniz     algebras  has 
dimension $13$  and it has 
five irreducible components (in particular, there is only one rigid algebra in this variety). In particular, it has no Vergne--Grunewald--O'Halloran Property (there is an irreducible component formed by only nilpotent algebras), but on the other hand, it has Vergne Property
(there are no rigid nilpotent algebras).
\end{theoremE}

\begin{theoremF}
The variety of  complex $5$-dimensional  nilpotent symmetric Leibniz     algebras  has 
dimension $24$  and it has six irreducible components 
(in particular, there are no rigid algebras in this variety.).
 
\end{theoremF}

\section{The algebraic classification of nilpotent  algebras}
\subsection{Method of classification of nilpotent algebras}
Throughout this paper, we use the notations and methods well written in \cite{  hac16},
which we have adapted for the weakly associative case with some modifications.
Further in this section we give some important definitions.

Let $({\bf A}, \cdot)$ be a complex  weakly associative   algebra 
and $\mathbb V$ be a complex  vector space. The $\mathbb C$-linear space ${\rm Z^{2}}\left(
\bf A,\mathbb V \right) $ is defined as the set of all  bilinear maps $\theta  \colon {\bf A} \times {\bf A} \longrightarrow {\mathbb V}$ such that
\[ \theta(xy,z)-\theta(x,yz)+\theta(yz,x)-\theta(y,zx)=\theta(yx,z)-\theta(y,xz). \]

These elements will be called {\it cocycles}. For a
linear map $f$ from $\bf A$ to  $\mathbb V$, if we define $\delta f\colon {\bf A} \times
{\bf A} \longrightarrow {\mathbb V}$ by $\delta f  (x,y ) =f(xy )$, then $\delta f\in {\rm Z^{2}}\left( {\bf A},{\mathbb V} \right) $. We define ${\rm B^{2}}\left({\bf A},{\mathbb V}\right) =\left\{ \theta =\delta f\ : f\in {\rm Hom}\left( {\bf A},{\mathbb V}\right) \right\} $.
We define the {\it second cohomology space} ${\rm H^{2}}\left( {\bf A},{\mathbb V}\right) $ as the quotient space ${\rm Z^{2}}
\left( {\bf A},{\mathbb V}\right) \big/{\rm B^{2}}\left( {\bf A},{\mathbb V}\right) $.

\

Let $\operatorname{Aut}({\bf A}) $ be the automorphism group of  ${\bf A} $ and let $\phi \in \operatorname{Aut}({\bf A})$. For $\theta \in
{\rm Z^{2}}\left( {\bf A},{\mathbb V}\right) $ define  the action of the group $\operatorname{Aut}({\bf A}) $ on ${\rm Z^{2}}\left( {\bf A},{\mathbb V}\right) $ by $\phi \theta (x,y)
=\theta \left( \phi \left( x\right) ,\phi \left( y\right) \right) $.  It is easy to verify that
 ${\rm B^{2}}\left( {\bf A},{\mathbb V}\right) $ is invariant under the action of $\operatorname{Aut}({\bf A}).$  
 So, we have an induced action of  $\operatorname{Aut}({\bf A})$  on ${\rm H^{2}}\left( {\bf A},{\mathbb V}\right)$.

\

Let $\bf A$ be a  weakly associative  algebra of dimension $m$ over  $\mathbb C$ and ${\mathbb V}$ be a $\mathbb C$-vector
space of dimension $k$. For the bilinear map $\theta$, define on the linear space ${\bf A}_{\theta } = {\bf A}\oplus {\mathbb V}$ the
bilinear product `` $\left[ -,-\right] _{{\bf A}_{\theta }}$'' by $\left[ x+x^{\prime },y+y^{\prime }\right] _{{\bf A}_{\theta }}=
 xy +\theta(x,y) $ for all $x,y\in {\bf A},x^{\prime },y^{\prime }\in {\mathbb V}$.
The algebra ${\bf A}_{\theta }$ is called a $k$-{\it dimensional central extension} of ${\bf A}$ by ${\mathbb V}$. One can easily check that ${\bf A_{\theta}}$ is a  weakly associative 
algebra if and only if $\theta \in {\rm Z^2}({\bf A}, {\mathbb V})$.

Call the
set $\operatorname{Ann}(\theta)=\left\{ x\in {\bf A}:\theta \left( x, {\bf A} \right)+ \theta \left({\bf A} ,x\right) =0\right\} $
the {\it annihilator} of $\theta $. We recall that the {\it annihilator} of an  algebra ${\bf A}$ is defined as
the ideal $\operatorname{Ann}(  {\bf A} ) =\left\{ x\in {\bf A}:  x{\bf A}+ {\bf A}x =0\right\}$. Observe
 that
$\operatorname{Ann}\left( {\bf A}_{\theta }\right) =(\operatorname{Ann}(\theta) \cap\operatorname{Ann}({\bf A}))
 \oplus {\mathbb V}$.

\

The following result shows that every algebra with a non-zero annihilator is a central extension of a smaller-dimensional algebra.

\begin{lemma}
Let ${\bf A}$ be an $n$-dimensional  weakly associative algebra such that $\dim (\operatorname{Ann}({\bf A}))=m\neq0$. Then there exists, up to isomorphism, a unique $(n-m)$-dimensional  weakly associative  algebra ${\bf A}'$ and a bilinear map $\theta \in {\rm Z^2}({\bf A'}, {\mathbb V})$ with $\operatorname{Ann}({\bf A'})\cap\operatorname{Ann}(\theta)=0$, where $\mathbb V$ is a vector space of dimension m, such that ${\bf A} \cong {{\bf A}'}_{\theta}$ and
 ${\bf A}/\operatorname{Ann}({\bf A})\cong {\bf A}'$.
\end{lemma}

\begin{proof}
Let ${\bf A}'$ be a linear complement of $\operatorname{Ann}({\bf A})$ in ${\bf A}$. Define a linear map $P \colon {\bf A} \longrightarrow {\bf A}'$ by $P(x+v)=x$ for $x\in {\bf A}'$ and $v\in\operatorname{Ann}({\bf A})$, and define a multiplication on ${\bf A}'$ by $[x, y]_{{\bf A}'}=P(x y)$ for $x, y \in {\bf A}'$.
For $x, y \in {\bf A}$, we have
\[P(xy)=P((x-P(x)+P(x))(y- P(y)+P(y)))=P(P(x) P(y))=[P(x), P(y)]_{{\bf A}'}. \]

Since $P$ is a homomorphism $P({\bf A})={\bf A}'$ is a  weakly associative algebra and
 ${\bf A}/\operatorname{Ann}({\bf A})\cong {\bf A}'$, which gives us the uniqueness. Now, define the map $\theta \colon {\bf A}' \times {\bf A}' \longrightarrow\operatorname{Ann}({\bf A})$ by $\theta(x,y)=xy- [x,y]_{{\bf A}'}$.
  Thus, ${\bf A}'_{\theta}$ is ${\bf A}$ and therefore $\theta \in {\rm Z^2}({\bf A'}, {\mathbb V})$ and $\operatorname{Ann}({\bf A'})\cap\operatorname{Ann}(\theta)=0$.
\end{proof}

\

\;
\begin{definition}
Let ${\bf A}$ be an algebra and $I$ be a subspace of $\operatorname{Ann}({\bf A})$. If ${\bf A}={\bf A}_0 \oplus I$
then $I$ is called an {\it annihilator component} of ${\bf A}$.
A central extension of an algebra $\bf A$ without annihilator component is called a {\it non-split central extension}.
\end{definition}

Our task is to find all central extensions of an algebra $\bf A$ by
a space ${\mathbb V}$.  In order to solve the isomorphism problem we need to study the
action of $\operatorname{Aut}({\bf A})$ on ${\rm H^{2}}\left( {\bf A},{\mathbb V}
\right) $. To do that, let us fix a basis $e_{1},\ldots ,e_{s}$ of ${\mathbb V}$, and $
\theta \in {\rm Z^{2}}\left( {\bf A},{\mathbb V}\right) $. Then $\theta $ can be uniquely
written as $\theta \left( x,y\right) =
\displaystyle \sum_{i=1}^{s} \theta _{i}\left( x,y\right) e_{i}$, where $\theta _{i}\in
{\rm Z^{2}}\left( {\bf A},\mathbb C\right) $. Moreover, $\operatorname{Ann}(\theta)=\operatorname{Ann}(\theta _{1})\cap\operatorname{Ann}(\theta _{2})\cap\ldots \cap\operatorname{Ann}(\theta _{s})$. Furthermore, $\theta \in
{\rm B^{2}}\left( {\bf A},{\mathbb V}\right) $\ if and only if all $\theta _{i}\in {\rm B^{2}}\left( {\bf A},
\mathbb C\right) $.
It is not difficult to prove (see \cite[Lemma 13]{hac16}) that given a  weakly associative algebra ${\bf A}_{\theta}$, if we write as
above $\theta \left( x,y\right) = \displaystyle \sum_{i=1}^{s} \theta_{i}\left( x,y\right) e_{i}\in {\rm Z^{2}}\left( {\bf A},{\mathbb V}\right) $ and 
$\operatorname{Ann}(\theta)\cap \operatorname{Ann}\left( {\bf A}\right) =0$, then ${\bf A}_{\theta }$ has an
annihilator component if and only if $\left[ \theta _{1}\right] ,\left[
\theta _{2}\right] ,\ldots ,\left[ \theta _{s}\right] $ are linearly
dependent in ${\rm H^{2}}\left( {\bf A},\mathbb C\right) $.

\;

Let ${\mathbb V}$ be a finite-dimensional vector space over $\mathbb C$. The {\it Grassmannian} $G_{k}\left( {\mathbb V}\right) $ is the set of all $k$-dimensional
linear subspaces of $ {\mathbb V}$. Let $G_{s}\left( {\rm H^{2}}\left( {\bf A},\mathbb C\right) \right) $ be the Grassmannian of subspaces of dimension $s$ in
${\rm H^{2}}\left( {\bf A},\mathbb C\right) $. There is a natural action of $\operatorname{Aut}({\bf A})$ on $G_{s}\left( {\rm H^{2}}\left( {\bf A},\mathbb C\right) \right) $.
Let $\phi \in \operatorname{Aut}({\bf A})$. For $W=\left\langle
\left[ \theta _{1}\right] ,\left[ \theta _{2}\right] ,\dots,\left[ \theta _{s}
\right] \right\rangle \in G_{s}\left( {\rm H^{2}}\left( {\bf A},\mathbb C
\right) \right) $ define $\phi W=\left\langle \left[ \phi \theta _{1}\right]
,\left[ \phi \theta _{2}\right] ,\dots,\left[ \phi \theta _{s}\right]
\right\rangle $. We denote the orbit of $W\in G_{s}\left(
{\rm H^{2}}\left( {\bf A},\mathbb C\right) \right) $ under the action of $\operatorname{Aut}({\bf A})$ by $\operatorname{Orb}(W)$. Given
\[
W_{1}=\left\langle \left[ \theta _{1}\right] ,\left[ \theta _{2}\right] ,\dots,
\left[ \theta _{s}\right] \right\rangle ,W_{2}=\left\langle \left[ \vartheta
_{1}\right] ,\left[ \vartheta _{2}\right] ,\dots,\left[ \vartheta _{s}\right]
\right\rangle \in G_{s}\left( {\rm H^{2}}\left( {\bf A},\mathbb C\right)
\right),
\]
we easily have that if $W_{1}=W_{2}$, then $ \bigcap\limits_{i=1}^{s}\operatorname{Ann}(\theta _{i})\cap \operatorname{Ann}\left( {\bf A}\right) = \bigcap\limits_{i=1}^{s}
\operatorname{Ann}(\vartheta _{i})\cap\operatorname{Ann}( {\bf A}) $, and therefore we can introduce
the set
\[
{\bf T}_{s}({\bf A}) =\left\{ W=\left\langle \left[ \theta _{1}\right] ,
\left[ \theta _{2}\right] ,\dots,\left[ \theta _{s}\right] \right\rangle \in
G_{s}\left( {\rm H^{2}}\left( {\bf A},\mathbb C\right) \right) : \bigcap\limits_{i=1}^{s}\operatorname{Ann}(\theta _{i})\cap\operatorname{Ann}({\bf A}) =0\right\},
\]
which is stable under the action of $\operatorname{Aut}({\bf A})$.

\

Now, let ${\mathbb V}$ be an $s$-dimensional linear space and let us denote by
${\bf E}\left( {\bf A},{\mathbb V}\right) $ the set of all {\it non-split $s$-dimensional central extensions} of ${\bf A}$ by
${\mathbb V}$. By above, we can write
\[
{\bf E}\left( {\bf A},{\mathbb V}\right) =\left\{ {\bf A}_{\theta }:\theta \left( x,y\right) = \sum_{i=1}^{s}\theta _{i}\left( x,y\right) e_{i} \ \ \text{and} \ \ \left\langle \left[ \theta _{1}\right] ,\left[ \theta _{2}\right] ,\dots,
\left[ \theta _{s}\right] \right\rangle \in {\bf T}_{s}({\bf A}) \right\} .
\]
We also have the following result, which can be proved as in \cite[Lemma 17]{hac16}.

\begin{lemma}
 Let ${\bf A}_{\theta },{\bf A}_{\vartheta }\in {\bf E}\left( {\bf A},{\mathbb V}\right) $. Suppose that $\theta \left( x,y\right) =  \displaystyle \sum_{i=1}^{s}
\theta _{i}\left( x,y\right) e_{i}$ and $\vartheta \left( x,y\right) =
\displaystyle \sum_{i=1}^{s} \vartheta _{i}\left( x,y\right) e_{i}$.
Then the  weakly associative algebras ${\bf A}_{\theta }$ and ${\bf A}_{\vartheta } $ are isomorphic
if and only if
$$\operatorname{Orb}\left\langle \left[ \theta _{1}\right] ,
\left[ \theta _{2}\right] ,\dots,\left[ \theta _{s}\right] \right\rangle =
\operatorname{Orb}\left\langle \left[ \vartheta _{1}\right] ,\left[ \vartheta
_{2}\right] ,\dots,\left[ \vartheta _{s}\right] \right\rangle .$$
\end{lemma}

This shows that there exists a one-to-one correspondence between the set of $\operatorname{Aut}({\bf A})$-orbits on ${\bf T}_{s}\left( {\bf A}\right) $ and the set of
isomorphism classes of ${\bf E}\left( {\bf A},{\mathbb V}\right) $. Consequently we have a
procedure that allows us, given a  weakly associative algebra ${\bf A}'$ of
dimension $n-s$, to construct all non-split central extensions of ${\bf A}'$. This procedure is:

\begin{enumerate}
\item For a given  weakly associative algebra ${\bf A}'$ of dimension $n-s $, determine ${\rm H^{2}}( {\bf A}',\mathbb {C}) $, $\operatorname{Ann}({\bf A}')$ and $\operatorname{Aut}({\bf A}')$.

\item Determine the set of $\operatorname{Aut}({\bf A}')$-orbits on ${\bf T}_{s}({\bf A}') $.

\item For each orbit, construct the  weakly associative algebra associated with a
representative of it.
\end{enumerate}

\medskip

The above described method gives all (commutative and non-commutative)  weakly associative algebras. But we are interested in developing this method in such a way that it only gives non-commutative  weakly associative   algebras, because the classification of all commutative  algebras is given in \cite{fkkv}. Clearly, any central extension of a non-commutative  weakly associative algebra is non-commutative. But a commutative algebra may have extensions which are not commutative algebras. More precisely, let ${\mathcal{C}}$ be a commutative algebra and $\theta \in {\rm Z_{\mathcal W}^2}({\mathcal{C}}, {\mathbb C}).$ Then ${\mathcal{C}}_{\theta }$ is a commutative algebra if and only if 
\begin{equation*}
 \theta(x,y)= \theta(y,x). 
 \end{equation*}
for all $x,y\in {\mathcal{C}}.$ Define the subspace ${\rm Z_{\mathcal{C}}^2}({\mathcal{C}},{\mathbb C})$ of ${\rm Z_{\mathcal{W}}^2}({\mathcal{C}},{\mathbb C})$ by
\begin{equation*}
{\rm Z_{\mathcal{C}}^2}({\mathcal{C}},{\mathbb C}) =\left\{\begin{array}{c} \theta \in {\rm Z_{\mathcal W}^2}({\mathcal{C}},{\mathbb C}) : \theta(x,y)=\theta(y,x) \text{ for all } x, y\in {\mathcal{C}}\end{array}\right\}.
\end{equation*}

Observe that ${\rm B^2}({\mathcal{C}},{\mathbb C})\subseteq{\rm Z_{\mathcal{C}}^2}({\mathcal{C}},{\mathbb C}).$
Let ${\rm H_{\mathcal{C}}^2}({\mathcal{C}},{\mathbb C}) =%
{\rm Z_{\mathcal{C}}^2}({\mathcal{C}},{\mathbb C}) \big/{\rm B^2}({\mathcal{C}},{\mathbb C}).$ Then ${\rm H_{\mathcal{C}}^2}({\mathcal{C}},{\mathbb C})$ is a subspace of $%
{\rm H_{\mathcal W}^2}({\mathcal{C}},{\mathbb C}).$ Define 
\begin{eqnarray*}
{\bf R}_{s}({\mathcal{C}})  &=&\left\{ {\mathcal W}\in {\bf T}_{s}({\mathcal{C}}) :{\mathcal W}\in G_{s}({\rm H_{\mathcal{C}}^2}({\mathcal{C}},{\mathbb C}) ) \right\}, \\
{\bf U}_{s}({\mathcal{C}})  &=&\left\{ {\mathcal W}\in {\bf T}_{s}({\mathcal{C}}) :{\mathcal W}\notin G_{s}({\rm H_{\mathcal{C}}^2}({\mathcal{C}},{\mathbb C}) ) \right\}.
\end{eqnarray*}
Then ${\bf T}_{s}({\mathcal{C}}) ={\bf R}_{s}(
{\mathcal{C}}) \mathbin{\mathaccent\cdot\cup} {\bf U}_{s}(
{\mathcal{C}}).$ The sets ${\bf R}_{s}({\mathcal{C}}) $
and ${\bf U}_{s}({\mathcal{C}})$ are stable under the action
of $\operatorname{Aut}({\mathcal{C}}).$ Thus, the  weakly associative  algebras
corresponding to the representatives of $\operatorname{Aut}({\mathcal{C}}) $%
-orbits on ${\bf R}_{s}({\mathcal{C}})$ are commutative  algebras,
while those corresponding to the representatives of $\operatorname{Aut}({\mathcal{C}}%
) $-orbits on ${\bf U}_{s}({\mathcal{C}})$ are non-commutative algebras. Hence, we may construct all non-split non-commutative  weakly associative algebras $%
\bf{A}$ of dimension $n$ with $s$-dimensional annihilator 
from a given weakly associative algebra $\bf{A}%
^{\prime }$ of dimension $n-s$ in the following way:

\begin{enumerate}
\item If $\bf{A}^{\prime }$ is non-commutative, then apply the procedure.

\item Otherwise, do the following:

\begin{enumerate}
\item Determine ${\bf U}_{s}(\bf{A}^{\prime })$ and $%
\operatorname{Aut}(\bf{A}^{\prime }).$

\item Determine the set of $\operatorname{Aut}(\bf{A}^{\prime })$-orbits on ${\bf U%
}_{s}(\bf{A}^{\prime }).$

\item For each orbit, construct the  weakly associative  algebra corresponding to one of its
representatives.
\end{enumerate}
\end{enumerate}

\begin{remark}
Let $\Omega$ be a variety of algebras defined by a family of polynomial identities $\omega$ and $\hat{\Omega}$  be a subvariety of $\Omega$ defined by a family of polynomial identities $\hat{\omega}.$ 
It is easy to see (see, for example, \cite{klp20}), that the presented procedure can be generalized for the search of all $n$-dimensional nilpotent algebras from $\Omega \setminus \hat{\Omega}.$ 
For example, in the present paper, we will find
\begin{enumerate}

\item  Weakly associative non-commutative algebras;
\item Weakly associative non-symmetric Leibniz algebras;
\item Symmetric Leibniz non-2-step nilpotent algebras.
\end{enumerate}

\end{remark}

\begin{remark}
The intersection of commutative algebras and (left, right or symmetric) Leibniz algebras gives the variety of $2$-step nilpotent commutative algebras.
\end{remark}
 
\subsubsection{Notations}
Let us introduce the following notations. Let ${\bf A}$ be a nilpotent algebra with
a basis $\{e_{1},e_{2}, \ldots, e_{n}\}.$ Then by $\Delta_{ij}$\ we will denote the
bilinear form
$\Delta_{ij}:{\bf A}\times {\bf A}\longrightarrow \mathbb C$
with $\Delta_{ij}(e_{l},e_{m}) = \delta_{il}\delta_{jm}.$
The set $\left\{ \Delta_{ij}:1\leq i, j\leq n\right\}$ is a basis for the linear space of
bilinear forms on ${\bf A},$ so every $\theta \in
{\rm Z^2}({\bf A},\bf \mathbb V )$ can be uniquely written as $
\theta = \displaystyle \sum_{1\leq i,j\leq n} c_{ij}\Delta _{{i}{j}}$, where $
c_{ij}\in \mathbb C$.
Let us fix the following notations for our nilpotent algebras:

$$\begin{array}{lll}
 {\mathcal W}_{j}& \mbox{---}& j\mbox{th }4\mbox{-dimensional   weakly associative  
 (non-commutative, non-symmetric Leibniz)  algebra.} \\

\mathcal{S}_{j}& \mbox{---}& j\mbox{th }4\mbox{-dimensional   symmetric Leibniz (non-2-step nilpotent) algebra.} \\

\mathbb{S}_{j}& \mbox{---}& j\mbox{th }5\mbox{-dimensional   symmetric Leibniz (non-2-step nilpotent) algebra.} \\

\mathcal{C}_{j}& \mbox{---}& j\mbox{th }3\mbox{-dimensional   commutative algebra.} \\

{\mathcal N}_{j}& \mbox{---}& j\mbox{th }3\mbox{-dimensional      non-commutative  $2$-step nilpotent algebra.} \\

{\mathfrak N}_{j}& \mbox{---}& j\mbox{th }4\mbox{-dimensional   $2$-step nilpotent algebra.}

\end{array}$$

\subsection{The algebraic classification of nilpotent 
weakly associative algebras}

\subsubsection{The algebraic classification of  $3$-dimensional nilpotent 
weakly associative algebras}
Thanks to \cite{fkkv}, we have the classification of all $3$-dimensional nilpotent algebras.
From this list we are choosing only $3$-dimensional nilpotent weakly associative algebras: 

\begin{longtable}{ll llllllllllll}
${\mathcal C}_{01}$ &$:$& $e_1 e_1 = e_2$\\
\hline${\mathcal C}_{02}$ &$:$&  $e_1 e_1 = e_2$  & $e_2 e_2=e_3$  \\
\hline${\mathcal C}_{03}$ &$:$& $e_1 e_1 = e_2$ & $e_1 e_2=e_3$ & $e_2 e_1=  e_3$\\
\hline${\mathcal C}_{04}$  &$:$& $e_1 e_2=e_3$  & $e_2 e_1=e_3$  \\
\hline${\mathcal N}_{01}$  &$:$& $e_1 e_2=e_3$  & $e_2 e_1=-e_3$   \\
\hline${\mathcal N}^{\lambda}_{02}$  &$:$& $e_1 e_1 = \lambda e_3$   & $e_2 e_1=e_3$  & $e_2 e_2=e_3$\end{longtable}


\subsubsection{The description of second cohomology spaces of $3$-dimensional nilpotent weakly associative algebras}\label{cohospace}

\
In the following table we give the description of the second cohomology space of  $3$-dimensional nilpotent weakly associative algebras.

\begin{longtable}{|l cr cl |  }
                \hline

${\rm H}^2_{\mathcal W}({\mathcal C}_{01})$& $=$& 
${\rm H}^2_{\mathcal{C}}({\mathcal C}_{01})$&$ \oplus$&$ \Big\langle [\Delta_{13}] \Big\rangle$ \\ && ${\rm H}^2_{\mathcal{C}}({\mathcal C}_{01}) $&$=$&$\Big\langle 
[\Delta_{12}+\Delta_{21}], [\Delta_{22}],[\Delta_{13}+\Delta_{31}], [\Delta_{23}+\Delta_{32}],  [\Delta_{33}]\Big\rangle$      \\ 

&& ${\rm H}^2_{\mathcal{S}}({\mathcal C}_{01}) $&$=$&$
\Big\langle 
[\Delta_{13}],[\Delta_{13}+\Delta_{31}],   [\Delta_{33}]\Big\rangle$      \\ \hline

${\rm H}^2_{\mathcal W}({\mathcal C}_{02})$& $=$& 
${\rm H}^2_{\mathcal{C}}({\mathcal C}_{02}) $&$=$&$\Big\langle 
[\Delta_{12}+\Delta_{21}],[\Delta_{13}+\Delta_{31}],   [\Delta_{32}+\Delta_{23}], [\Delta_{33}]
\Big\rangle$    \\ \hline

${\rm H}^2_{\mathcal W}({\mathcal C}_{03})$& $=$& 
${\rm H}^2_{\mathcal{C}}({\mathcal C}_{03}) $&$=$&$\Big\langle 
[\Delta_{22}], [\Delta_{13}+\Delta_{31}],  [\Delta_{32}+\Delta_{23}], [\Delta_{33}]
\Big\rangle$     \\ \hline

${\rm H}^2_{\mathcal W}({\mathcal C}_{04})$& $=$& 
${\rm H}^2_{\mathcal{C}}({\mathcal C}_{04}) $&$\oplus$&$ \Big\langle [\Delta_{12}]  \Big\rangle$ \\&&$ 
{\rm H}^2_{\mathcal{C}}({\mathcal C}_{04}) $&$=$&$\Big\langle
[\Delta_{11}], [\Delta_{22}],  [\Delta_{13}+\Delta_{31}],   [\Delta_{32}+\Delta_{23}], [\Delta_{33}]\Big\rangle$  

 \\&&$ 
{\rm H}^2_{\mathcal{S}}({\mathcal C}_{04}) $&$=$&$\Big\langle
[\Delta_{11}], [\Delta_{12}], [\Delta_{22}] \Big\rangle$

    \\ \hline

${\rm H}^2_{\mathcal W}({\mathcal N}_{01})$& $=$& 
${\rm H}^2_{\mathcal S}({\mathcal N}_{01}) $&$=$&$\Big\langle 
[\Delta_{11}], [\Delta_{12}], [\Delta_{22}],
[\Delta_{13}-\Delta_{31}],
[\Delta_{23}-\Delta_{32}]

\Big\rangle$   \\ \hline

${\rm H}^2_{\mathcal W}({\mathcal N}^{\lambda}_{02})$& $=$& 
${\rm H}^2_{\mathcal{S}}({\mathcal N}^{\lambda}_{02}) $& $=$&$ \Big\langle 
[\Delta_{11}], [\Delta_{12}], [\Delta_{21}]
\Big\rangle$ 
 
    \\ \hline

 \end{longtable}

\subsubsection{Central extensions of ${\mathcal C}_{01}$}
	Let us use the following notations:
\begin{longtable}{lll}
$\nabla_1 =[\Delta_{12}+\Delta_{21}],$& $\nabla_2 = [\Delta_{22}],$& $ \nabla_3 = [\Delta_{13}+\Delta_{31}],$\\
$\nabla_4 = [\Delta_{23}+\Delta_{32}],$& $\nabla_5 = [\Delta_{33}],$& $\nabla_6 =[\Delta_{13}].$
\end{longtable}
Take $\theta=\sum\limits_{i=1}^6\alpha_i\nabla_i\in {\rm H_{{\mathcal W}}^2}({\mathcal C}_{01}).$ 
	The automorphism group of ${\mathcal C}_{01}$ consists of invertible matrices of the form
\begin{center}$	\phi=
	\begin{pmatrix}
	x &    0  &  0\\
	z &  x^2  &  r\\
	t &  0  &  y \\
	\end{pmatrix}.$\end{center}
	Since
	$$
	\phi^T\begin{pmatrix}
	0 &  \alpha_1 & \alpha_3+\alpha_6\\
	\alpha_1  &  \alpha_2 & \alpha_4 \\
	\alpha_3&  \alpha_4    & \alpha_5\\
	\end{pmatrix} \phi=	\begin{pmatrix}
		\alpha^*   &  \alpha_1^* & \alpha_3^*+\alpha_6^*\\
	\alpha_1^*  & \alpha_2^* & \alpha_4^* \\
	\alpha_3^*&  \alpha_4^*    & \alpha_5^*
	\end{pmatrix},
	$$
	 we have that the action of ${\rm Aut} ({\mathcal C}_{01})$ on the subspace
$\langle \sum\limits_{i=1}^6\alpha_i\nabla_i  \rangle$
is given by
$\langle \sum\limits_{i=1}^6\alpha_i^{*}\nabla_i\rangle,$
where
\begin{longtable}{ll}
$\alpha^*_1=   ( \alpha_1x + \alpha_2z + \alpha_4t)x^2,$ &
$\alpha^*_2=   \alpha_2x^4,$ \\
$\alpha^*_3=  (\alpha_1x+ \alpha_2z+  \alpha_4t)r+ ( \alpha_3x+ \alpha_4z+ \alpha_5t)y,$ &
$\alpha_4^*=  ( \alpha_2r+ \alpha_4y)x^2,$\\
$\alpha_5^*= \alpha_2 r^2+2  \alpha_4 r y+ \alpha_5y^2,$&
$\alpha_6^*= \alpha_6 x y .$
\end{longtable}

Since ${\rm H^2_{{\mathcal W}}({\mathcal C}_{01})} = {\rm H^2_{{\mathcal C}}({\mathcal C}_{01})}\oplus \langle \nabla_6\rangle$ 
and we are interested only in new algebras, we have $(\alpha_1,\alpha_2,\alpha_4)\neq 0$ and  $\alpha_6\neq0.$
Then 
\begin{enumerate}
    \item if $\alpha_2 \neq 0$ and $\alpha_4^2\neq \alpha_2 \alpha_5,$ then by choosing 
    $x=\frac{\alpha_6}{\sqrt{\alpha_2 \alpha_5-\alpha_4^2}},$
    $y=\frac{\alpha_2 \alpha_6^2}{\sqrt{(\alpha_2 \alpha_5-\alpha_4^2)^3}},$
    $z=\frac{(\alpha_3 \alpha_4-\alpha_1 \alpha_5) \alpha_6}{\sqrt{(\alpha_2 \alpha_5-\alpha_4^2)^3}},$
    $t=\frac{(\alpha_1 \alpha_4-\alpha_2 \alpha_3) \alpha_6}{\sqrt{(\alpha_2 \alpha_5-\alpha_4^2)^3}}$
    and 
    $r=-\frac{\alpha_4 \alpha_6^2}{\sqrt{(\alpha_2 \alpha_5-\alpha_4^2)^3}},$
    we have the representative $\langle \nabla_2+ \nabla_5+ \nabla_6 \rangle.$
    
   \item if $\alpha_2 \neq 0$ and $\alpha_4^2= \alpha_2 \alpha_5,$ then by choosing 
     $x=\alpha_6,$ 
     $y=\alpha_2 \alpha_6^2,$
     $z=-\frac{\alpha_1 \alpha_6}{\alpha_2},$
     $t=0$
     and 
     $r=-\alpha_4 \alpha_6^2,$
    we have the family of representatives    $\langle \nabla_2+ \alpha \nabla_3+ \nabla_6 \rangle.$   
    
    \item if $\alpha_2=0$ and $\alpha_4\neq 0,$ by choosing 
    $x=\frac{\alpha_6}{\alpha_4},$
    $y=2 \alpha_4,$
    $z=\frac{\alpha_1 \alpha_5 \alpha_6-\alpha_3 \alpha_4 \alpha_6}{\alpha_4^3},$
    $t=-\frac{\alpha_1 \alpha_6}{\alpha_4^2}$
    and 
    $r=-\alpha_5,$
    we have the representative   $\langle \nabla_4+ \nabla_6 \rangle.$   
    
    \item if $\alpha_2=\alpha_4=0$ and $\alpha_5\neq 0, \alpha_1\neq 0,$ by choosing 
    $x=\frac{\alpha_6^2}{\alpha_1 \alpha_5},$
    $y=\frac{\alpha_6^3}{\alpha_1 \alpha_5^2},$
    $z=0,$
    $t=-\frac{\alpha_3 \alpha_6^2}{\alpha_1 \alpha_5^2}$
    and $r=0,$ we have the representative  $\langle \nabla_1+\nabla_5+ \nabla_6 \rangle.$ 


  \item if $\alpha_2=\alpha_4=\alpha_5=0$ and $\alpha_1\neq 0,$ by choosing 
    $x=\alpha_6,$
    $y=\alpha_1\alpha_6,$
    $z=0,$
    $t=0$  
    and $r=-\alpha_3\alpha_6$, we have the representative  $\langle  \nabla_1+ \nabla_6 \rangle.$ 
 
\end{enumerate}
Summarizing, we have the following distinct orbits 
\begin{center} 
$\langle \nabla_1+\nabla_5+ \nabla_6 \rangle,$ \  
$\langle  \nabla_1+ \nabla_6 \rangle,$ \  
$\langle \nabla_2+ \alpha \nabla_3+ \nabla_6 \rangle,$ \\
$\langle \nabla_2+ \nabla_5+ \nabla_6 \rangle,$ \ 
$\langle \nabla_4+ \nabla_6 \rangle,$
\end{center}
which give the following new algebras:

 \begin{longtable}{lllllll}
$\mathcal{W}_{01}$ &$:$& $ e_1e_1=e_2$ & $e_1 e_2 =e_4$ & $e_1e_3=e_4$ & $e_2 e_1 =e_4$ & $e_3e_3 =e_4$\\
\hline
$\mathcal{W}_{02}$ &$:$& $ e_1e_1=e_2$ & $e_1 e_2 =e_4$ & $e_1e_3=e_4$ & $e_2 e_1 =e_4$ & \\
\hline
$\mathcal{W}_{03}^\alpha$ &$:$& $ e_1e_1=e_2$ & $e_1e_3=(\alpha+1)e_4$ & $e_2e_2=e_4$ & $e_3 e_1 =\alpha e_4$\\
\hline
$\mathcal{W}_{04}$ &$:$& $ e_1e_1=e_2$ & $e_1e_3=e_4$ & $e_2e_2=e_4$ & $e_3e_3 =e_4$\\
\hline
$\mathcal{W}_{05}$ &$:$& $ e_1e_1=e_2$ & $e_1e_3=e_4$ & $e_2e_3=e_4$ & $e_3e_2 =e_4$\\

    \end{longtable}

\subsubsection{Central extensions of ${\mathcal C}_{04}$}
	Let us use the following notations:
\begin{longtable}{lll}
$\nabla_1 =[\Delta_{11}],$& $\nabla_2 = [\Delta_{22}],$& $ \nabla_3 = [\Delta_{13}+\Delta_{31}],$\\
$\nabla_4 = [\Delta_{23}+\Delta_{32}],$& $\nabla_5 = [\Delta_{33}],$& $\nabla_6 =[\Delta_{12}].$
\end{longtable}
Take $\theta=\sum\limits_{i=1}^6\alpha_i\nabla_i\in {\rm H_{{\mathcal W}}^2}({\mathcal C}_{04}).$ 
	The automorphism group of ${\mathcal C}_{04}$ consists of invertible matrices of the form
\begin{center}	
$	\phi_1=
	\begin{pmatrix}
	x &    0  &  0\\
	0 &  y  &  0\\
	z &  w  &  xy \\
	\end{pmatrix}$ and 
$	\phi_2=
	\begin{pmatrix}
	0 &    x  &  0\\
	y &  0  &  0\\
	z &  w  &  xy \\
	\end{pmatrix}$	
	
	\end{center}
	Since
	$$
	\phi_1^T\begin{pmatrix}
	\alpha_1 &  \alpha_6 & \alpha_3 \\
	 0   &  \alpha_2 & \alpha_4 \\
	\alpha_3&  \alpha_4    & \alpha_5\\
	\end{pmatrix} \phi_1=	\begin{pmatrix}
		\alpha_1^*   &  \alpha^*+\alpha_6^* & \alpha_3^* \\
	\alpha^*  & \alpha_2^* & \alpha_4^* \\
	\alpha_3^*&  \alpha_4^*    & \alpha_5^*
	\end{pmatrix},
	$$
	 we have that the action of ${\rm Aut} ({\mathcal C}_{04})$ on the subspace
$\langle \sum\limits_{i=1}^6\alpha_i\nabla_i  \rangle$
is given by
$\langle \sum\limits_{i=1}^6\alpha_i^{*}\nabla_i\rangle,$
where
\begin{longtable}{ll}
$\alpha^*_1=  \alpha_1x^2 +2  \alpha_3x z + \alpha_5z^2,$ &
$\alpha^*_2=  \alpha_2y^2 +2  \alpha_4w y + \alpha_5w^2,$ \\
$\alpha^*_3= (\alpha_3x+ \alpha_5z)x y, $ &
$\alpha_4^*= (\alpha_4y+ \alpha_5w)x y,$\\
$\alpha_5^*= \alpha_5 x^2 y^2,$&
$\alpha_6^*= \alpha_6 x y. $
\end{longtable}

Since ${\rm H^2_{{\mathcal W}}({\mathcal C}_{04})} = {\rm H^2_{{\mathcal C}}({\mathcal C}_{04})}\oplus \langle \nabla_6\rangle$ 
and we are interested only in new algebras, we have $(\alpha_3,\alpha_4,\alpha_5)\neq 0$ and   $\alpha_6\neq0.$
Then 
\begin{enumerate}
\item if $\alpha_5\neq 0$ and $\alpha_4^2 \neq \alpha_2 \alpha_5,$ by choosing $x=\frac{\sqrt{\alpha_2 \alpha_5-\alpha_4^2}}{\alpha_5},$
$y=\frac{\alpha_6}{\sqrt{\alpha_2 \alpha_5-\alpha_4^2}},$
$z=-\frac{\alpha_3 \sqrt{\alpha_2 \alpha_5-\alpha_4^2}}{\alpha_5^2}$
and 
$w=-\frac{\alpha_4 \alpha_6}{\alpha_5 \sqrt{\alpha_2 \alpha_5-\alpha_4^2}},$
we have the family of representatives 
 $\langle \alpha \nabla_1+ \nabla_2+ \nabla_5+ \nabla_6 \rangle.$ 

\item if $\alpha_5\neq 0, \alpha_4^2 = \alpha_2 \alpha_5$ and $\alpha_3^2\neq \alpha_1 \alpha_5,$ by choosing 
$x=\frac{\alpha_6}{\sqrt{\alpha_1 \alpha_5-\alpha_3^2}},$
$y=\frac{\sqrt{\alpha_1 \alpha_5-\alpha_3^2}}{\alpha_5},$
$z=-\frac{\alpha_3 \alpha_6}{\alpha_5 \sqrt{\alpha_1 \alpha_5-\alpha_3^2}}$
and 
$w=-\frac{\alpha_4 \sqrt{\alpha_1 \alpha_5-\alpha_3^2}}{\alpha_5^2},$
we have the    representative  
 $\langle  \nabla_1+ \nabla_5+ \nabla_6 \rangle.$ 

\item if $\alpha_5\neq 0, \alpha_4^2 = \alpha_2 \alpha_5$ and $\alpha_3^2= \alpha_1 \alpha_5,$ by choosing 
$x= \alpha_5,$
$y=\frac{\alpha_6}{\alpha_5^2},$
$z=-\alpha_3$
and 
$w=-\frac{\alpha_4 \alpha_6}{\alpha_5^3},$
we have the    representative  
 $\langle   \nabla_5+ \nabla_6 \rangle.$ 

\item if $\alpha_5= 0$ and $\alpha_4\neq0,  \alpha_3 \neq 0,$ by choosing 
$x= \frac{\alpha_6}{\alpha_3},$
$y= \frac{\alpha_6}{\alpha_4},$
$z= -\frac{\alpha_1 \alpha_6}{2 \alpha_3^2}$
and 
$w= -\frac{\alpha_2 \alpha_6}{2 \alpha_4^2},$
we have the    representative  
 $\langle   \nabla_3+\nabla_4+ \nabla_6 \rangle.$

\item if $\alpha_5=\alpha_4= 0$ and $\alpha_3\neq0, \alpha_2\neq 0,$ by choosing 
$x= \frac{\alpha_6}{\alpha_3},$
$y= \frac{\alpha_6^2}{\alpha_2 \alpha_3},$
$z= -\frac{\alpha_1 \alpha_6}{2 \alpha_3^2}$
and 
$w= 0,$
we have the    representative  
 $\langle   \nabla_2+\nabla_3+ \nabla_6 \rangle.$

\item if $\alpha_5=\alpha_4=\alpha_2= 0$ and $\alpha_3\neq0,$ by choosing 
$x= \frac{\alpha_6}{\alpha_3},$
$y= 1,$
$z= -\frac{\alpha_1 \alpha_6}{2 \alpha_3^2}$
and 
$w= 0,$
we have the     representative  
 $\langle   \nabla_3+ \nabla_6 \rangle.$ 
 
\item if $\alpha_5=\alpha_3= 0$ and $\alpha_4\neq0,$ by taking $\phi_2$
we have   $ \alpha_3^*\neq 0.$ The case is considered above.  
  
  
 
 
\end{enumerate}
Summarizing, we have the following distinct orbits 
\begin{center}
 $\langle \alpha \nabla_1+ \nabla_2+ \nabla_5+ \nabla_6 \rangle,$ \ 
  $\langle  \nabla_1+ \nabla_5+ \nabla_6 \rangle,$ \ 
     $\langle   \nabla_2+\nabla_3+ \nabla_6 \rangle,$ \\ 
    $\langle   \nabla_3+\nabla_4+ \nabla_6 \rangle,$ \ 
      $\langle   \nabla_3+ \nabla_6 \rangle,$ \  
   $\langle   \nabla_5+ \nabla_6 \rangle,$ 
\end{center}      
which give the following new algebras:

 \begin{longtable}{llllllll}
$\mathcal{W}_{06}^{\alpha}$ &$:$& $ e_1e_1=\alpha e_4$ & $e_1 e_2 =e_3+e_4$ & $e_2 e_1 =e_3$ & $e_2e_2=e_4$ & $e_3e_3 =e_4$\\
\hline
$\mathcal{W}_{07}$ &$:$& $ e_1e_1=e_4$ & $e_1 e_2 =e_3+e_4$ & $e_2 e_1 =e_3$ & $e_3e_3 =e_4$\\
\hline
$\mathcal{W}_{08}$ &$:$& $e_1 e_2 =e_3+e_4$ & $e_1e_3=e_4$ & $e_2 e_1 =e_3$ & $e_2e_2=e_4$ & $e_3e_1=e_4$\\
\hline
$\mathcal{W}_{09}$ &$:$& $e_1 e_2 =e_3+e_4$ & $e_1e_3=e_4$ & $e_2 e_1 =e_3$ & $e_2e_3=e_4$ & $e_3e_1=e_4$ & $e_3e_2=e_4$\\
\hline
$\mathcal{W}_{10}$ &$:$& $e_1 e_2 =e_3+e_4$ & $e_1e_3=e_4$ & $e_2 e_1 =e_3$ & $e_3e_1=e_4$ \\
\hline
$\mathcal{W}_{11}$ &$:$& $e_1 e_2 =e_3+e_4$ & $e_2 e_1 =e_3$ & $e_3e_3 =e_4$\\

    \end{longtable}

\subsubsection{Central extensions of ${\mathcal N}_{01}$}\label{n01-orb}
	Let us use the following notations:
\begin{longtable}{lll}
$\nabla_1 =[\Delta_{11}],$& $\nabla_2 = [\Delta_{12}],$& $ \nabla_3 = [\Delta_{22}],$\\
$\nabla_4 = [\Delta_{13}-\Delta_{31}],$& $\nabla_5 = [\Delta_{23}-\Delta_{32}].$ 
\end{longtable}
Take $\theta=\sum\limits_{i=1}^5\alpha_i\nabla_i\in {\rm H_{{\mathcal W}}^2}({\mathcal N}_{01}).$ 
	The automorphism group of ${\mathcal N}_{01}$ consists of invertible matrices of the form
\begin{center}	
$	\phi=
	\begin{pmatrix}
	x &    y  &  0\\
	z &  w  &  0\\
	t &  r  &  xw-zy \\
	\end{pmatrix}$  
	
	\end{center}
	Since
	$$
	\phi^T\begin{pmatrix}
	\alpha_1 &  \alpha_2 & \alpha_4 \\
	0  &  \alpha_3 & \alpha_5 \\
	-\alpha_4&  -\alpha_5    & 0\\
	\end{pmatrix} \phi=	\begin{pmatrix}
		\alpha_1^*   &  -\alpha^*+\alpha_2^* & \alpha_4^* \\
	\alpha^*  & \alpha_3^* & \alpha_5^* \\
	-\alpha_4^*&  -\alpha_5^*    & 0
	\end{pmatrix},
	$$
	 we have that the action of ${\rm Aut} ({\mathcal N}_{01})$ on the subspace
$\langle \sum\limits_{i=1}^5\alpha_i\nabla_i  \rangle$
is given by
$\langle \sum\limits_{i=1}^5\alpha_i^{*}\nabla_i\rangle,$
where
\begin{longtable}{ll}
$\alpha^*_1=   \alpha_1x^2 + \alpha_2 x z+ \alpha_3z^2,$ &
$\alpha^*_2= 2  \alpha_1x y + \alpha_2w x + \alpha_2y z+2 \alpha_3 w z,$ \\
$\alpha^*_3=  \alpha_1y^2+ \alpha_2w y+ \alpha_3w^2,$ &
$\alpha_4^*=  (w x-y z) ( \alpha_4x + \alpha_5z),$\\&
$\alpha_5^*=(w x-y z) (  \alpha_4 y + \alpha_5w).$ 
\end{longtable}

Since, we are interested only in new algebras, we have  $(\alpha_4, \alpha_5)\neq0$
and we can suppose that $\alpha_4=0$ and $\alpha_5\neq0.$
Then 
\begin{enumerate}
\item if $\alpha_1 \neq 0$ and $\alpha_2^2\neq4 \alpha_1 \alpha_3,$ by choosing 
$x=\frac{4 \alpha_1 \alpha_3-\alpha_2^2}{4 \alpha_1 \alpha_5},$
$y=-\frac{\alpha_2 \sqrt{4 \alpha_1 \alpha_3-\alpha_2^2}}{4 \alpha_1 \alpha_5},$
$z=0,$
$w=\frac{\sqrt{4 \alpha_1 \alpha_3-\alpha_2^2}}{2 \alpha_5},$
$r=0$
and
$t=0,$
we have the representative 
 $\langle     \nabla_1+\nabla_3+\nabla_5 \rangle.$

\item if $\alpha_1 \neq 0$ and $\alpha_2^2= 4\alpha_1 \alpha_3,$ by choosing 
$x= 4 \alpha_1 \alpha_5,$
$y=-\alpha_2,$
$z=0,$
$w=2 \alpha_1,$
$r=0$
and
$t=0,$
we have the representative 
 $\langle     \nabla_1+  \nabla_5 \rangle.$

 \item if $\alpha_1 = 0$ and $\alpha_2\neq 0,$ by choosing 
$x=1,$
$y=-\frac{\alpha_3}{\alpha_5},$
$z=0,$
$w=\frac{\alpha_2}{\alpha_5},$
$r=0$
and
$t=0,$
we have the representative 
 $\langle     \nabla_2+  \nabla_5 \rangle.$

  \item if $\alpha_1= \alpha_2= 0$ and $\alpha_3\neq  0,$ by choosing 
$x=\frac{\alpha_3}{\alpha_5},$
$y=0,$
$z=0,$
$w=1,$
$r=0$
and
$t=0,$
we have the representative 
 $\langle     \nabla_3+  \nabla_5 \rangle.$
 
   \item if $\alpha_1= \alpha_2= \alpha_3=  0,$
we have the representative 
 $\langle   \nabla_5 \rangle.$

 \end{enumerate}

 Summarizing, we have the following distinct orbits 
\begin{center}
  $\langle     \nabla_1+\nabla_3+\nabla_5 \rangle,$ \ 
  $\langle     \nabla_1+  \nabla_5 \rangle,$ \ 
   $\langle     \nabla_2+  \nabla_5 \rangle,$ \
  $\langle     \nabla_3+  \nabla_5 \rangle,$ \
    $\langle   \nabla_5 \rangle,$
 \end{center}      
which give the following new algebras:
 
 \begin{longtable}{llllllll}
$\mathcal{S}_{01}$ &$:$& $ e_1e_1=e_4$ & $e_1 e_2 =e_3$ & $e_2 e_1 =-e_3$ & $e_2e_2=e_4$ & $e_2e_3 =e_4$ & $e_3e_2=-e_4$\\
\hline
$\mathcal{S}_{02}$ &$:$& $ e_1e_1=e_4$ & $e_1 e_2 =e_3$ & $e_2 e_1 =-e_3$ & $e_2e_3 =e_4$ & $e_3e_2=-e_4$\\
\hline
$\mathcal{S}_{03}$ &$:$& $e_1 e_2 =e_3+e_4$ & $e_2 e_1 =-e_3$ & $e_2e_3 =e_4$ & $e_3e_2=-e_4$\\
\hline
$\mathcal{S}_{04}$ &$:$& $e_1 e_2 =e_3$ & $e_2 e_1 =-e_3$ & $e_2e_2=e_4$ & $e_2e_3 =e_4$ & $e_3e_2=-e_4$\\
\hline
$\mathcal{S}_{05}$ &$:$& $e_1 e_2 =e_3$ & $e_2 e_1 =-e_3$ & $e_2e_3 =e_4$ & $e_3e_2=-e_4$\\
    \end{longtable}

\subsubsection{Classification theorems for $4$-dimensional algebras}\label{secteoA}
Now we are ready summarize all results related to the algebraic classification of complex $4$-dimensional nilpotent weakly associative and symmetric Leibniz algebras.

\begin{theoremA}\label{teorA}
Let ${\mathcal S}$ be a complex  $4$-dimensional nilpotent symmetric Leibniz algebra.
Then ${\mathcal S}$ is a $2$-step nilpotent algebra, described in \cite{kppv}; 
or it is isomorphic to one algebra from the following list:

\begin{center}
 
\begin{longtable}{llllllll}
$\mathcal{S}_{01}$ &$:$& $ e_1e_1=e_4$ & $e_1 e_2 =e_3$ & $e_2 e_1 =-e_3$ & $e_2e_2=e_4$ & $e_2e_3 =e_4$ & $e_3e_2=-e_4$\\
\hline
$\mathcal{S}_{02}$ &$:$& $ e_1e_1=e_4$ & $e_1 e_2 =e_3$ & $e_2 e_1 =-e_3$ & $e_2e_3 =e_4$ & $e_3e_2=-e_4$\\
\hline
$\mathcal{S}_{03}$ &$:$& $e_1 e_2 =e_3+e_4$ & $e_2 e_1 =-e_3$ & $e_2e_3 =e_4$ & $e_3e_2=-e_4$\\
\hline
$\mathcal{S}_{04}$ &$:$& $e_1 e_2 =e_3$ & $e_2 e_1 =-e_3$ & $e_2e_2=e_4$ & $e_2e_3 =e_4$ & $e_3e_2=-e_4$\\
\hline
$\mathcal{S}_{05}$ &$:$& $e_1 e_2 =e_3$ & $e_2 e_1 =-e_3$ & $e_2e_3 =e_4$ & $e_3e_2=-e_4$\\
    \end{longtable}
\end{center}

Let ${\mathcal W}$ be a complex  $4$-dimensional nilpotent weakly associative  algebra.
Then ${\mathcal W}$ is a commutative algebra, described in \cite{fkkv}; 
or a non-commutative  symmetric Leibniz algebra; 
or it is isomorphic to one algebra from the following list:

\begin{center}

\begin{longtable}{llllllll}
$\mathcal{W}_{01}$ &$:$& $ e_1e_1=e_2$ & $e_1 e_2 =e_4$ & $e_1e_3=e_4$ & $e_2 e_1 =e_4$ & $e_3e_3 =e_4$\\
\hline
$\mathcal{W}_{02}$ &$:$& $ e_1e_1=e_2$ & $e_1 e_2 =e_4$ & $e_1e_3=e_4$ & $e_2 e_1 =e_4$ & \\
\hline
$\mathcal{W}_{03}^\alpha$ &$:$& $ e_1e_1=e_2$ & $e_1e_3=(\alpha+1)e_4$ & $e_2e_2=e_4$ & $e_3 e_1 =\alpha e_4$\\
\hline
$\mathcal{W}_{04}$ &$:$& $ e_1e_1=e_2$ & $e_1e_3=e_4$ & $e_2e_2=e_4$ & $e_3e_3 =e_4$\\
\hline
$\mathcal{W}_{05}$ &$:$& $ e_1e_1=e_2$ & $e_1e_3=e_4$ & $e_2e_3=e_4$ & $e_3e_2 =e_4$\\
\hline
$\mathcal{W}_{06}^{\alpha}$ &$:$& $ e_1e_1=\alpha e_4$ & $e_1 e_2 =e_3+e_4$ & $e_2 e_1 =e_3$ & $e_2e_2=e_4$ & $e_3e_3 =e_4$\\
\hline
$\mathcal{W}_{07}$ &$:$& $ e_1e_1=e_4$ & $e_1 e_2 =e_3+e_4$ & $e_2 e_1 =e_3$ & $e_3e_3 =e_4$\\
\hline
$\mathcal{W}_{08}$ &$:$& $e_1 e_2 =e_3+e_4$ & $e_1e_3=e_4$ & $e_2 e_1 =e_3$ & $e_2e_2=e_4$ & $e_3e_1=e_4$\\
\hline
$\mathcal{W}_{09}$ &$:$& $e_1 e_2 =e_3+e_4$ & $e_1e_3=e_4$ & $e_2 e_1 =e_3$ & $e_2e_3=e_4$ & $e_3e_1=e_4$ & $e_3e_2=e_4$\\
\hline
$\mathcal{W}_{10}$ &$:$& $e_1 e_2 =e_3+e_4$ & $e_1e_3=e_4$ & $e_2 e_1 =e_3$ & $e_3e_1=e_4$ \\
\hline
$\mathcal{W}_{11}$ &$:$& $e_1 e_2 =e_3+e_4$ & $e_2 e_1 =e_3$ & $e_3e_3 =e_4$\\

    \end{longtable}
\end{center}

\end{theoremA}

The list of complex $4$-dimensional (non-nilpotent, non-Lie, one sided) Leibniz algebras is given in \cite{ikv17}.
After a carefully verification, we can obtain the following list of symmetric Leibniz algebras.

\begin{theoremB}\label{teorB}
Let ${\bf S}$ be a complex  $4$-dimensional  symmetric Leibniz (non-nilpotent, non-Lie) algebra.
Then ${\bf S}$   is isomorphic to one algebra from the following list:

\begin{center} 
\begin{longtable}{llllllllllll} 

$\mathfrak{L}_{02}$ &$:$ & $
  e_1e_1=e_4$ & $e_1e_2=-e_2$ & $e_1e_3=e_3$ & $e_2e_1=e_2$ \\
  && $e_2e_3=e_4$ &  $e_3e_1=-e_3$ &  $e_3e_2=-e_4$\\ 
\hline

$\mathfrak{L}_{14}$ &$:$ & $ 
e_1e_3=-e_3$& $e_2e_2=e_4$ & $e_3e_1=e_3$ \\ 
\hline

$\mathfrak{L}_{15}^{\alpha}$ &$:$ & $  
e_1e_1=\alpha e_4$ & $e_1e_2=e_4$ & $e_1e_3=-e_3$ &  $e_2e_2=e_4$ &   $e_3e_1=e_3$\\ 
\hline

$\mathfrak{L}_{16} $ &$:$ & $ 
e_1e_1=-2e_4$ &  $e_1e_3=-e_3$ &   $e_2e_2=e_4$ & $e_3e_1=e_3$\\ 
\hline

$\mathfrak{L}_{24}^{\alpha}$ &$:$ & $ 
e_1e_1=e_4$ & $e_1e_2= - e_2$ &  $ e_1e_3= -\alpha e_3$ & $e_2e_1=e_2$ & $e_3e_1= \alpha e_3$\\ 
\hline

$\mathfrak{L}_{27} $ &$:$ & $ 
e_1e_2= -e_2$ &  $e_1e_3= e_4$ & $e_2e_1=e_2$ \\ 
\hline

$\mathfrak{L}_{30}$ &$:$ & $ 
e_1e_1 =e_4$& $e_1e_2=-e_2-e_3$ & $e_1e_3=- e_3$ &  $e_2e_1=e_2+e_3$ & $e_3e_1= e_3$\\ 
\hline

$\mathfrak{L}_{35}^{\alpha \neq-1}$ &$:$ & $ 
e_1e_2=\alpha e_3$ & $e_1e_4=- e_4$ & $e_2e_1=e_3$ & $e_4e_1= e_4$\\ 
\hline

$\mathfrak{L}_{36}$ &$:$ & $
e_1e_1 =e_3$&  $e_1e_2=-e_3$ & $e_1e_4=- e_4$ & $e_2e_1=e_3$ & $e_4e_1= e_4$     \\ 
\hline

\end{longtable}
\end{center}

\end{theoremB}

\subsection{The algebraic classification of symmetric Leibniz algebras}
The algebraic classification of complex $5$-dimensional nilpotent symmetric Leibniz algebras consists of two parts:
\begin{enumerate}
    \item $5$-dimensional algebras with identity $xyz=0$ (also known as $2$-step nilpotent algebras) is the intersection of all varieties of algebras defined by a family of polynomial identities   of degree three or more; for example, it is in intersection of associative, Zinbiel, Leibniz, etc, algebras. All these algebras can be obtained as central extensions of zero-product algebras. The geometric classification of $2$-step nilpotent algebras is given in \cite{ikp20}. It is the reason why we are not interested in it.
    
    \item $5$-dimensional nilpotent symmetric Leibniz (non-$2$-step nilpotent) algebras, which are central extensions of nilpotent symmetric Leibniz algebras with non-zero product of a smaller dimension.
  \end{enumerate}
   We will give a complete classification of complex $5$-dimensional nilpotent non-split  symmetric Leibniz algebras. By  \cite[Theorem 2.2]{saidsuper}, we must only consider central extensions of Lie algebras by a symmetric Leibniz cocycle.  Thanks to the section \ref{cohospace}, 
   the present classification is based on the calculation of $2$-dimensional central extensions of the algebra ${\mathcal N}_{01},$ given in the section \ref{N2-2dim}, and the calculation of $1$-dimensional central extensions of $4$-dimensional symmetric Leibniz algebras, listed in the section \ref{coho-4dim}.

\subsubsection{ 
 $2$-dimensional central extensions of ${\mathcal N}_{01}$}
 \label{N2-2dim}
Let $\xi_1$ and $\xi_2$ be two generators of our $2$-dimensional cohomology subspace.
Thanks to the section \ref{n01-orb}, we can suppose that 
\begin{center}$\xi_1 \in \left\{   \langle     \nabla_1+\nabla_3+\nabla_5 \rangle,
\langle     \nabla_1+  \nabla_5 \rangle,
\langle     \nabla_2+  \nabla_5 \rangle, 
\langle     \nabla_3+  \nabla_5 \rangle,
    \langle   \nabla_5 \rangle \right\}$
    and $\xi_2 =\sum\limits_{i=1}^4 \beta_i \nabla_i,$ $\beta_4\in \{0,1\}.$
    \end{center}
Below, we will consider $10$ cases, related  to these restrictions. For some cases we will prove that they are reduced to other cases, and for the rest of cases, we will find all non-isomorphic representatives.

\begin{enumerate}
    \item $\xi_1=\nabla_1+\nabla_3+\nabla_5,$ $\xi_2=\sum\limits_{i=1}^3\beta_i\nabla_i+\nabla_4.$

\begin{enumerate}
\item if $\beta_1 \neq \frac{4 \beta_2 - 4 \beta_2^2 + \beta_2^3 - 8 \beta_3^2}{
 4 (\beta_2-2) \beta_3}, \beta_3\neq 0, \beta_2 \neq 2,$ by choosing 
\begin{center}
$x=(\beta_1+\beta_3+R) ((3 \beta_2^2-4-4 \beta_2) \beta_3+4 \beta_3^3+\beta_1 ((\beta_2-2)^2-4 \beta_3^2)+(\beta_2-2)^2 R+4 \beta_3^2 R),$
$y=\beta_1 (\beta_2-2)+(\beta_2+2) \beta_3+(\beta_2-2) R,$
$z=2 ((3 \beta_2^2-4 \beta_2-4) \beta_3+4 \beta_3^3+\beta_1 ((\beta_2-2)^2-4 \beta_3^2)+(\beta_2-2)^2 R+4 \beta_3^2 R),$
$w=\beta_2^2-2 \beta_2-2 \beta_3 (\beta_1-\beta_3-R),$  \ where  
$R=\sqrt{\beta_2^2+( \beta_1 -\beta_3)^2},$
\end{center}
we have a representative $\langle \nabla_1+\nabla_5, \sum\limits_{i=1}^3\beta_i^{\star}\nabla_i+\nabla_4 \rangle.$ 
This case will be considered below.

\item if $\beta_1 = \frac{4 \beta_2 - 4 \beta_2^2 + \beta_2^3 - 8 \beta_3^2}{
 4 (\beta_2-2) \beta_3}, \beta_3\neq 0, \beta_2 \neq 2,$ by choosing 
 $x= \frac{2 z \beta_3}{2-\beta_2}$ and 
 $y=\frac{2-\beta_2-2 w z \beta_3}{z ( \beta_2-2)},$
 we have  representatives 
 \begin{center}$\langle \nabla_3+\nabla_5,\sum\limits_{i=1}^3\beta_i^{\star}\nabla_i+\nabla_4 \rangle$
or 
$\langle \nabla_5, \sum\limits_{i=1}^3\beta_i^{\star}\nabla_i+\nabla_4 \rangle,$ \end{center}
depending on $\beta_2=0$ or not.
This case will be considered below.

\item if $\beta_3=0, \beta_2\not\in \{ 0,2\}$ and $(\beta_1,\beta_2)\neq 0,$
by choosing 
$x=T z$
and
$y=\frac{T w (\beta_2-2)}{2 T^2-2 T \beta_1-\beta_2},$
where $4 T^2 -4 T \beta_1 -\beta_2^2=0,$ we have representatives 
\begin{center}
$\langle \nabla_1+\nabla_5, \sum\limits_{i=1}^3\beta_i^{\star}\nabla_i+\nabla_4 \rangle,$ $\langle \nabla_3+\nabla_5, \sum\limits_{i=1}^3\beta_i^{\star}\nabla_i+\nabla_4 \rangle$
or 
$\langle \nabla_5, \sum\limits_{i=1}^3\beta_i^{\star}\nabla_i+\nabla_4 \rangle,$
\end{center}
depending on $\beta_2=0$ or not.
This case will be considered below.

\item if $\beta_3=0, \beta_2=0,$
by choosing 
$x=0,$
$y=1,$
and
$w=1$ we have representatives 
\begin{center}$\langle \nabla_3+\nabla_5, \sum\limits_{i=1}^3\beta_i^{\star}\nabla_i+\nabla_4 \rangle$
or 
$\langle \nabla_5,\sum\limits_{i=1}^3\beta_i^{\star}\nabla_i+\nabla_4 \rangle,$
\end{center}depending on $\beta_1=0$ or not.
This case will be considered below.

\item if $\beta_3=0, \beta_2=2,$
by choosing 
$x=1,$
$y=0,$
$z=\frac{2}{\beta_1+\sqrt{4+\beta_1^2}}$
and
$w=1,$ we have a representative 
$\langle \nabla_3+ \nabla_5, \sum\limits_{i=1}^3\beta_i^{\star}\nabla_i+\nabla_4 \rangle.$
This case will be considered below.

\item if $\beta_3\neq 0, \beta_2=2,$
by choosing 
$x=\frac{\beta_1+\beta_3+R}{2},$
$y=\frac{\sqrt{2} \beta_3}{\sqrt{\beta_3^2 (\beta_3-\beta_1+R)}},$
$z=1$
and
$w=\frac{\sqrt{\beta_3^2 (\beta_3-\beta_1+R)}}{\sqrt{2} \beta_3},$
where $R=\sqrt{4+(\beta_1- \beta_3)^2},$
we have a representative 
$\langle \nabla_1+ \nabla_5, \sum\limits_{i=1}^3\beta_i^{\star}\nabla_i+\nabla_4 \rangle.$
This case will be considered below.

\end{enumerate}

   \item $\xi_1=\nabla_1+\nabla_3+\nabla_5,$ $\xi_2=\sum\limits_{i=1}^3\beta_i \nabla_i.$

\begin{enumerate}
  
    \item if $\beta_1\neq0$ and $\beta_2\neq 0,$ by choosing 
    $x=1,$ 
    $y=\frac{\beta_3-\beta_1}{\beta_1},$
    $z=0$
    and
    $w=-\frac{\beta_2}{\beta_1},$
    we have a representative
$\langle \nabla_2+ \nabla_5, \sum\limits_{i=1}^3\beta_i^{\star}\nabla_i \rangle.$
This case will be considered below.

  \item if $\beta_1\neq0$ and $\beta_2= 0,$ by choosing 
    $x=1,$ 
    $y=0,$
    $z=0$
    and
    $w=1,$
    we have representatives
$\langle \nabla_3+ \nabla_5, \sum\limits_{i=1}^3\beta_i^{\star}\nabla_i \rangle$
or
$\langle  \nabla_5, \sum\limits_{i=1}^3\beta_i^{\star}\nabla_i \rangle,$ depending on $\beta_1=\beta_3$ or not. This case will be considered below.

  \item if $\beta_1=0$ and $\beta_2\neq 0,$ by choosing 
    $x=\beta_2^2,$ 
    $y=-\beta_3+\sqrt{\beta_2^2+\beta_3^2},$
    $z=0$
    and
    $w= \beta_2,$
    we have a representative
$\langle \nabla_1+ \nabla_5, \sum\limits_{i=1}^3\beta_i^{\star}\nabla_i\rangle.$
This case will be considered below.

  \item if $\beta_1=0$ and $\beta_2= 0,$ by choosing 
    $x=1,$ 
    $y=0,$
    $z=0$
    and
    $w= 1,$
    we have a representative
$\langle \nabla_1+ \nabla_5,\sum\limits_{i=1}^3\beta_i^{\star}\nabla_i\rangle.$
This case will be considered below.

\end{enumerate}

\item $\xi_1=\nabla_1+\nabla_5,$ $\xi_2=\sum\limits_{i=1}^3 \beta_i \nabla_i+\nabla_4.$

\begin{enumerate}
    \item $\beta_3\neq0,$ then by choosing 
    $x=1,$    $w=-\frac{y(\beta_2+z \beta_3)}{\beta_3}$    and 
    $z,$ such that $z^3 \beta_3+z^2 \beta_2+z \beta_1=1,$ we have 
    the representative 
 \begin{center}   $\left\langle 
    \frac{\beta_3 (2 \beta_3+2 z^2 \beta_2 \beta_3+2 z^3 \beta_3^2+z (\beta_2^2-2 \beta_1 \beta_3))}{y (\beta_2+2 z \beta_3)^2} \nabla_2+ \nabla_5, \sum\limits_{i=1}^3\beta_i^{\star}\nabla_i+\nabla_4\right\rangle,$
    \end{center}
which give one from  the following representatives
\begin{center}$\langle 
     \nabla_2+ \nabla_5, \sum\limits_{i=1}^3\beta_i^{\star}\nabla_i+\nabla_4\rangle$
     or 
     $\langle 
     \nabla_5, \sum\limits_{i=1}^3\beta_i^{\star}\nabla_i+\nabla_4\rangle,$\end{center}
     depending on $2 \beta_3+2 z^2 \beta_2 \beta_3+2 z^3 \beta_3^2+z (\beta_2^2-2 \beta_1 \beta_3)=0$ or not. This case will be considered below.

\item $\beta_3=0,$ then by choosing $x=0,$ $y=1$ and $z=1,$ we have the representative 
\begin{center}   $\langle 
     -\beta_2  \nabla_2-( \beta_1+w \beta_2) \nabla_3+ \nabla_5, \sum\limits_{i=1}^3\beta_i^{\star}\nabla_i+\nabla_4\rangle,$
    \end{center}
which give one from the following representatives
\begin{center}$\langle 
     \xi , \sum\limits_{i=1}^3\beta_i^{\star}\nabla_i+\nabla_4\rangle, \ \xi \in \{ \nabla_2+\nabla_5, \nabla_3+\nabla_5, \nabla_5 \}, $
     \end{center}
depending on $\beta_2=0$,   $\beta_1=0$  or not. This case will be considered below.

\end{enumerate}

    \item $\xi_1=\nabla_1+\nabla_5,$ $\xi_2=\sum\limits_{i=1}^3 \beta_i\nabla_i.$ 
    
    \begin{enumerate}
        \item if $\beta_1\neq 0,$ by choosing 
        $x=1,$
         $z=0$
        and 
        $w=1,$
        we have the representative
\begin{center}        $\langle -\frac{\beta_2}{\beta_1}\nabla_2-\frac{y \beta_2+\beta_3}{\beta_1} \nabla_3+\nabla_5,  \sum\limits_{i=1}^3\beta_i^{\star}\nabla_i\rangle,$ \end{center} 
   which give one of the representatives
\begin{center}$\langle 
     \xi , \sum\limits_{i=1}^3\beta_i^{\star}\nabla_i+\nabla_4\rangle, \ \xi \in \{ \nabla_2+\nabla_5, \nabla_3+\nabla_5, \nabla_5 \}, $
     \end{center}
depending on $\beta_2=0, \beta_3=0$  or not.    This case will be considered below.

        \item if $\beta_1= 0, \beta_2\neq 0$ and $\beta_3\neq 0,$ by choosing 
        $x=\beta_3^2\beta_2^{-2},$
        $y=0,$
        $z=0$
        and 
        $w=\beta_3\beta_2^{-1},$ 
        we have the  representative
        $\langle \nabla_1+\nabla_5,  \nabla_2+ \nabla_3   \rangle.$ 

    \item if $\beta_1= 0, \beta_2= 0$ and $\beta_3\neq 0,$  
        we have the     representative 
        $\langle \nabla_1+\nabla_5,     \nabla_3   \rangle.$

 \item if $\beta_1= 0, \beta_2\neq 0$ and $\beta_3= 0,$  
        we have the     representative 
        $\langle \nabla_1+\nabla_5,     \nabla_2   \rangle.$ 

 \item if $\beta_1\neq 0, \beta_2= 0$ and $\beta_3= 0,$  
        we have the     representative 
        $\langle \nabla_5,     \nabla_1   \rangle.$

    \end{enumerate}

\item $\xi_1=\nabla_3+\nabla_5,$ $\xi_2=\sum\limits_{i=1}^3 \beta_i \nabla_i+\nabla_4.$

\begin{enumerate}
    \item $\beta_1\neq 0$ and $4 \beta_1\beta_3 \neq  (1 -  \beta_2)^2,$
    by choosing 
    $x= \beta_1-\beta_1 \beta_2-R,$
    $y=\frac{R-\beta_1-\beta_1 \beta_2}{2 \beta_1},$
    $z=2 \beta_1^2$
    and 
    $w=\beta_1,$
    where $R=\beta_1\sqrt{(1- \beta_2)^2-4 \beta_1 \beta_3},$
    we have the representative 
$\langle  \nabla_2+\nabla_5, 
\sum\limits_{i=1}^3\beta^{\star}_i\nabla_i+\nabla_4\rangle.$
    This case will be considered below.

    \item $\beta_1\neq 0$ and $4 \beta_1\beta_3 =  1 - 2 \beta_2 + \beta_2^2,$
    by choosing 
    $x=1-\beta_2,$
    $y=-\frac{\beta_2+1}{2},$ 
    $z=2\beta_1$
    and 
    $w=\beta_1,$
    we have the representative $\langle  \nabla_2+\nabla_5, 
\sum\limits_{i=1}^3\beta^{\star}_i\nabla_i+\nabla_4\rangle.$
   This case will be considered below.

   \item 

$\beta_1=0, \beta_2\neq 1$ and $\beta_3\neq 0,$ by choosing 
   $x=1,$ 
   $z=\frac{ 1-\beta_2}{ \beta_3}$
   and $w=0,$ we have  one from the following representatives 
\begin{center}$\langle    \nabla_2+\nabla_5, 
\sum\limits_{i=1}^3\beta^{\star}_i\nabla_i+\nabla_4\rangle$
or 
$\langle  \nabla_5, 
\sum\limits_{i=1}^3\beta^{\star}_i\nabla_i+\nabla_4\rangle,$\end{center}
depending on $\beta_2=0$ or not.
These cases will be considered below.

 \item 

$\beta_1=0, \beta_2= 1$ and $\beta_3\neq 0,$ by choosing 
   $x=1,$ 
   $y=0,$
   $z=0$
   and $w= \beta_3^{-1},$ we have  the following representative
$\langle    \nabla_3+\nabla_5, 
 \nabla_2+\nabla_3+\nabla_4\rangle.$

  \item 
  
$\beta_1=0, \beta_2 \neq 0$ and $\beta_3= 0,$ by choosing 
   $x=0,$ 
   $y=-\beta_2,$
   $z=1$
   and 
   $w=0,$
   we have the representative 
$\langle  \nabla_2+ \nabla_5, 
\sum\limits_{i=1}^3\beta^{\star}_i\nabla_i+\nabla_4\rangle.$ 
This case will be considered below.

  \item 
  
$\beta_1=0, \beta_2 = 0$ and $\beta_3= 0,$ by choosing 
   $x=0,$ 
   $y=-1,$
   $z=1$
   and 
   $w=1,$
   we have the representative 
$\langle    \nabla_5, 
\sum\limits_{i=1}^3\beta^{\star}_i\nabla_i+\nabla_4\rangle.$ 
This case will be considered below.

\end{enumerate}

\item $\xi_1=\nabla_3+\nabla_5,$ $\xi_2=\sum\limits_{i=1}^3 \beta_i \nabla_i.$

\begin{enumerate}
\item $\beta_1 = 0$ and $(\beta_2, \beta_3)\neq 0,$ by choosing $z=0,$
we have the representative 
\begin{center}
$\left\langle  -\frac{ \beta_2}{y\beta_2+w\beta_3} \nabla_2 +\nabla_5, \sum\limits_{i=1}^3 \beta_i \nabla_i \right\rangle,$\end{center}
which gives one of the following representatives
\begin{center}
$\langle  \nabla_2 +\nabla_5, \sum\limits_{i=1}^3 \beta_i^{\star} \nabla_i\rangle$
or
$\langle \nabla_5, \sum\limits_{i=1}^3 \beta_i^{\star} \nabla_i\rangle,$
\end{center}
depending on $\beta_2=0$ or not. This case will be considered below.

\item $\beta_1 \neq 0,$ by choosing 
$x=1,$ $y=-\frac{w\beta_2}{2\beta_1}$ and $z=0,$
we have the representative 
\begin{center}$\langle \nabla_3+\nabla_5,  4\beta_1^2 \nabla_1+ (4\beta_1 \beta_3- \beta_2^2)w^2 \nabla_3 \rangle,$\end{center}
which gives the representatives 
\begin{center}$\langle \nabla_3+\nabla_5,    \nabla_1+  \nabla_3 \rangle$ or
$\langle \nabla_3+\nabla_5,    \nabla_1 \rangle,$\end{center}
depending on $4\beta_1 \beta_3= \beta_2^2$ or not.
\end{enumerate}

   \item $\xi_1=\nabla_2+\nabla_5,$  $\xi_2=\sum\limits_{i=1}^3 \beta_i \nabla_i+\nabla_4,$ then by choosing 
      $y=0,$  $z=0$ and $w=1$ we have 
\begin{center}    $\langle \nabla_2+\nabla_5, 
    \beta_1 \nabla_1+ \beta_2 x^{-1}\nabla_2+ \beta_3 x^{-2}\nabla_3+\nabla_4\rangle,$\end{center} which gives 
\begin{center}    
$\langle \nabla_2+\nabla_5, \xi \rangle, 
\xi \in \{
    \alpha \nabla_1+ \beta \nabla_2+ \nabla_3+\nabla_4, \ 
    \alpha \nabla_1+  \nabla_2+ \nabla_4, \ 
    \nabla_1+  \nabla_4, \ \nabla_4\},$ \end{center}
depending on $\beta_i=0 \ ( 1 \leq i \leq 3)$ or not.

    \item $\xi_1=\nabla_2+\nabla_5,$ $ \xi_2= \sum\limits_{i=1}^3\beta_i \nabla_i.$ 
    
    \begin{enumerate}
        \item if $\beta_1\neq 0, \beta_2\neq 0$ and $\beta_3\neq 0,$ by choosing 
        $x=\sqrt{\beta_3\beta_1^{-1}}$, 
        $y=0,$
        $z=0$
        and 
        $w=1,$
        we have the family of representatives
        $\langle \nabla_2+\nabla_5,  \nabla_1+\alpha \nabla_2 +\nabla_3 \rangle.$

        \item if $\beta_1\neq 0, \beta_2\neq 0$ and $\beta_3= 0,$ by choosing 
        $x=\beta_2\beta_1^{-1},$
        $y=0,$
        $z=0$
        and 
        $w=1,$
        we have the  representative
        $\langle \nabla_2+\nabla_5,  \nabla_1+ \nabla_2   \rangle.$ 

        \item if $\beta_1\neq 0, \beta_2= 0$ and $\beta_3\neq 0,$ by choosing 
        $x=\sqrt{\beta_3\beta_1^{-1}},$
        $y=0,$
        $z=0$
        and 
        $w=1,$
        we have the  representative
        $\langle \nabla_2+\nabla_5,  \nabla_1+ \nabla_3   \rangle.$ 

  \item if $\beta_1= 0, \beta_2\neq 0$ and $\beta_3\neq 0,$ by choosing 
       $x=\beta_3\beta_2^{-1},$ 
        $y=0,$
        $z=0$
        and 
        $w=1,$
        we have the  representative
        $\langle \nabla_2+\nabla_5,  \nabla_2+ \nabla_3   \rangle.$ 
    
       \item if $\beta_1= 0, \beta_2= 0$ and $\beta_3\neq 0,$  
        we have the     representative 
        $\langle \nabla_2+\nabla_5,     \nabla_3   \rangle.$

 \item if $\beta_1= 0, \beta_2\neq 0$ and $\beta_3= 0,$  
        we have the     representative 
        $\langle \nabla_5,     \nabla_2   \rangle.$ 

 \item if $\beta_1\neq 0, \beta_2= 0$ and $\beta_3= 0,$  
        we have the     representative 
        $\langle \nabla_2+\nabla_5,     \nabla_1   \rangle.$

\end{enumerate}

\item $\xi_1=\nabla_5, \xi_2= \sum\limits_{i=1}^3\beta_i \nabla_i+\nabla_4.$

\begin{enumerate}
    \item $\beta_1\neq 0,$ by choosing 
    $y=-\frac{\beta_2}{2},$ and
    $z=0$ and 
    $w=\beta_1,$ we have the representative 
    \begin{center}
         $\left\langle  \nabla_5,  \nabla_1-\frac{\beta_2^2-4\beta_1\beta_3}{4 x^2}\nabla_3+\nabla_4\right\rangle,$
    \end{center}
    which gives one of the following representatives
 \begin{center}   $\left\langle  \nabla_5,  \nabla_1+\nabla_3+\nabla_4\right\rangle$
    or
        $\left\langle \nabla_5,   \nabla_1+\nabla_4\right\rangle,$
\end{center}
depending on $4\beta_1\beta_3=\beta_2^2$ or not.
    \item $\beta_1=0,$ by choosing 
        $z=0$
    and 
    $w=1,$ we have the representative
      \begin{center}
         $\left\langle  \nabla_5,  \frac{\beta_2}{x}\nabla_2+\frac{y\beta_2+\beta_3}{x^2}\nabla_3+\nabla_4\right\rangle,$
    \end{center}
    which gives one of the following representatives
\begin{center}
$\left\langle  \nabla_5,  \nabla_2+\nabla_4\right\rangle$, 
        $\left\langle  \nabla_5,   \nabla_3+\nabla_4\right\rangle,$  or $\left\langle  \nabla_5,   \nabla_4\right\rangle,$
        \end{center}
depending on $\beta_2=0$,  $\beta_3=0$  or not.
\end{enumerate}

\item $\xi_1=\nabla_5, \xi_2= \sum\limits_{i=1}^3\beta_i \nabla_i.$

\begin{enumerate}
    \item $\beta_1\neq 0,$ by choosing 
    $y=-\frac{\beta_2}{2\beta_1},$ 
    $z=0$ and 
    $w=1,$ we have the representative 
  \begin{center}
         $\left\langle  \nabla_5,  x^2\beta_1\nabla_1+\frac{ 4\beta_1\beta_3-\beta_2^2}{4\beta_1}\nabla_3\right\rangle,$
    \end{center} 
    which gives one of the following representatives
\begin{center}   $\left\langle  \nabla_5,  \nabla_1+\nabla_3\right\rangle$
    or
        $\left\langle \nabla_5,   \nabla_1\right\rangle,$
\end{center}
depending on $4\beta_1\beta_3=\beta_2^2$ or not.
    \item $\beta_1=0,$ by choosing 
    $x=1,$
    $z=0$
    and 
    $w=1,$ we have the representative
      \begin{center}
         $\left\langle  \nabla_5,  \beta_2\nabla_2+(y\beta_2+\beta_3)\nabla_3\right\rangle,$
    \end{center}
    which gives one of the following representatives
$\left\langle \nabla_5,  \nabla_2\right\rangle$
    or
        $\left\langle  \nabla_5,   \nabla_3\right\rangle,$
depending on $\beta_2=0$ or not.
    
\end{enumerate}
\end{enumerate}

Consider the orbit $\langle\nabla_2+\nabla_5,\alpha\nabla_1+\nabla_2+\nabla_4\rangle:$
\begin{enumerate}
    \item if $\alpha\neq0$ then by taking $x=0$, $y=-1$, $z=i\sqrt{\alpha}$ and $w=\alpha$, we obtain the representative \begin{center}$\langle\nabla_2+\nabla_5,\frac{i(1+\sqrt{\alpha})}{\sqrt{\alpha}}\nabla_2+\nabla_3+\nabla_4\rangle$;
    \end{center}
    
    \item if $\alpha=0$ then by taking $x=\frac{i}{2}$, $y=-\frac{1}{2}$, $z=\frac{i}{2}$, $w=\frac{1}{2}$, we obtain the representative \begin{center}
        $\langle\nabla_5,\nabla_1+\nabla_3+\nabla_4\rangle$.
    \end{center}
\end{enumerate}

Consider now the orbit $\langle\nabla_5,\nabla_2+\nabla_4\rangle:$ 

By taking $x=0$, $y=-1$, $z=-1$ and $w=0$, we obtain the representative $\langle\nabla_2+\nabla_5,\nabla_4\rangle$. 

Summarizing, we obtain the following distinct orbits:
\begin{center}
$    \left\langle \nabla_1+\nabla_5,\nabla_2+\nabla_3\right\rangle, 
    \left\langle \nabla_1+\nabla_5,\nabla_3\right\rangle, 
    \left\langle \nabla_1+\nabla_5,\nabla_2\right\rangle,  
    \left\langle \nabla_5,\nabla_1\right\rangle,  \left\langle \nabla_3+\nabla_5,\nabla_2+\nabla_3+\nabla_4\right\rangle,$
    \\
$   
    \left\langle \nabla_3+\nabla_5,\nabla_1+\nabla_3\right\rangle, 
    \left\langle \nabla_3+\nabla_5,\nabla_1\right\rangle,    \left\langle \nabla_2+\nabla_5,\nabla_1+\nabla_4\right\rangle, 
    \left\langle \nabla_2+\nabla_5,\nabla_4\right\rangle, $ 
    \\
$ 
    \left\langle \nabla_2+\nabla_5,\nabla_1+\alpha\nabla_2+\nabla_3\right\rangle^{O(\alpha)=O(-\alpha)},   \left\langle \nabla_2+\nabla_5,\nabla_1+\nabla_2\right\rangle, 
    \left\langle \nabla_2+\nabla_5,\nabla_2+\nabla_3\right\rangle, $ 
    \\
$ 
    \left\langle \nabla_2+\nabla_5,\nabla_3\right\rangle, 
    \left\langle \nabla_5,\nabla_2\right\rangle,    \left\langle \nabla_2+\nabla_5,\nabla_1\right\rangle, 
    \left\langle \nabla_5,\nabla_1+\nabla_3+\nabla_4\right\rangle, $ 
    \\
$ 
    \left\langle \nabla_5,\nabla_1+\nabla_4\right\rangle,   \left\langle \nabla_5,\nabla_3+\nabla_4\right\rangle, \left\langle \nabla_5,\nabla_4\right\rangle, \left\langle \nabla_5,\nabla_1+\nabla_3\right\rangle$
\end{center}
and a special family of orbits 
\begin{center}    ${\mathfrak O}(\alpha,\beta)=\left\langle  \nabla_2+\nabla_5,\alpha\nabla_1+\beta\nabla_2+\nabla_3+\nabla_4\right\rangle,$
\end{center} 
 which will be considered below:
 \begin{enumerate}
     \item 
It is easy to see that    ${\mathfrak O}(\alpha,\beta)={\mathfrak O}(\alpha,-\beta).$

\item\label{supercaso} If $\alpha\neq 0,1$ and 
$\beta \neq \pm \frac{1 - 2 \alpha}{\sqrt{\alpha}},$ then by choosing 
$x=\frac{(\beta Q+\alpha+Q^2) \sqrt{Q^2-\alpha}}{Q(Q^2+\alpha+\beta Q-1)},$
$y=\frac{Q^2-\alpha}{Q (Q^2+\alpha+\beta Q-1)},$
$z=\frac{\sqrt{Q^2-\alpha}}{Q^2+\alpha+\beta Q-1},$
$w=\frac{Q^2-\alpha}{Q^2+\alpha+\beta Q-1}$
 and 
$Q =\frac{ \alpha \beta \pm \alpha\sqrt{ 4 - 4 \alpha +\beta^2}}{2 - 2 \alpha},$
we have the representative
\begin{center}${\mathfrak O}\Big( \frac{ \alpha^2+\beta Q+\alpha (1+\beta Q+Q^2)}{Q^2-\alpha },
\frac{(1+2 \alpha+\beta Q) (Q^2+\alpha+\beta Q-1)}{(Q^2+\alpha+\beta Q-1)\sqrt{Q^2-\alpha}  }\Big).$
\end{center}

\item If $\alpha=1, \beta \neq 0, \pm 1,$ then by choosing 
\begin{center}$x=\frac{(\beta Q+1+Q^2) \sqrt{Q^2-1}}{Q^2 (\beta+Q)},$
$y=\frac{Q^2-1}{Q^2 (\beta+Q)},$
$z=\frac{\sqrt{Q^2-1}}{Q (\beta+Q)},$
$w=\frac{Q^2-1}{Q (\beta+Q)},$
and $Q=-\beta^{-1},$
\end{center}
we have the representative
${\mathfrak O}\Big( \frac{2+2 \beta Q+Q^2}{Q^2-1},
\frac{Q (\beta+Q) (3+\beta Q)}{ Q (\beta+Q) \sqrt{Q^2-1}}   \Big),$
which will be jointed with the   case (\ref{supercaso}).

\item If $\beta =   \frac{1 - 2 \alpha}{\sqrt{\alpha}}$
(if $\beta =   -\frac{1 - 2 \alpha}{\sqrt{\alpha}}$ we have the isomorphic case)
and $\alpha\neq 0, \frac{1}{2},1,$ then by choosing 
\begin{center}$x=\frac{i\alpha}{ \sqrt{1-2 \alpha}},$
$y=\frac{\alpha-1}{\sqrt{\alpha}},$
$z=\frac{i(\alpha-1)\sqrt{\alpha}}{\sqrt{1-2\alpha}}$ and
$w=\alpha,$
\end{center}
we have the representative
${\mathfrak O}\Big( \frac{\alpha}{2 \alpha-1},
-\frac{i}{ \sqrt{\alpha(1-2\alpha)}}   \Big),$
which will be jointed with the     case (\ref{supercaso}).

\item If $\alpha=\frac{1}{2}, \beta=0,$ 
then ${\mathfrak O}(\frac{1}{2}, 0)$ is not isomorphic  to
other
${\mathfrak O}(\gamma_1,\gamma_2).$

\item If $\alpha=0,$ then ${\mathfrak O}(0,\beta)$ is isomorphic only to 
${\mathfrak O}(0,-\beta).$

 \end{enumerate}

\subsubsection{  The description of second cohomology spaces of $4$-dimensional nilpotent Lie algebras }\label{coho-4dim}

By Theorem 2.2 in \cite{saidsuper}, we must only consider complex $4$-dimensional nilpotent Lie algebras. In the next table we give the list of complex $4$-dimensional nilpotent Lie algebras and their cohomology spaces.

\

\begin{longtable}{llllllll} 
\hline

${\mathfrak n}_{3}$ &$:$ &  $e_1e_2=  e_3$ & $e_2e_1=-e_3$  \\ 

\multicolumn{8}{l}{
${\rm H}^2_{\mathbb{S}}({\mathfrak n}_{3})=
\Big\langle 
[\Delta_{ 1 1}], [\Delta_{ 1 2}], 
[\Delta_{ 1 4}], [\Delta_{ 2 2}], 
[\Delta_{ 2 4}], 
[\Delta_{ 1 3} - \Delta_{ 3 1}], 
[\Delta_{ 2 3} - \Delta_{ 3 2}], [\Delta_{ 4 1}], 
[\Delta_{ 4 2}], [\Delta_{ 4 4}]
\Big\rangle $}\\

\hline
${\mathfrak n}_{4}$ &$:$& $e_1 e_2 =e_3$ & $e_2 e_1 =-e_3$ & $e_2e_3 =e_4$ & $e_3e_2=-e_4$\\
  
\multicolumn{8}{l}{
${\rm H}^2_{\mathbb{S}}({\mathfrak n}_{4})=
\Big\langle 
[\Delta_{ 1 1}], [\Delta_{ 1 2}], 
[\Delta_{ 2 2}], [\Delta_{ 1 3} - 
\Delta_{ 3 1}], [\Delta_{ 2 4} - 
\Delta_{ 4 2}]
\Big\rangle $}\\
\hline
\end{longtable}

\subsubsection{Central extensions of ${\mathfrak n}_{3}$}
	Let us use the following notations:
\begin{longtable}{lllll}

$\nabla_1 =[\Delta_{11}],$& $\nabla_2 =[\Delta_{ 1 2}],$ & $\nabla_3 =[\Delta_{ 2 2}],$&
$\nabla_4 =[\Delta_{ 1 3}] -  [\Delta_{ 3 1}],$ &
$\nabla_5 =[\Delta_{ 2 3}] -  [\Delta_{ 3 2}],$ \\
$\nabla_6 =[\Delta_{ 1 4}],$&
$\nabla_7 =[\Delta_{ 4 1}],$&
$\nabla_8 =[\Delta_{ 2 4}],$& 
$\nabla_9 =[\Delta_{ 4 2}],$& 
$\nabla_{10} =[\Delta_{ 4 4}].$
\end{longtable}
Take $\theta=\sum\limits_{i=1}^{10}\alpha_i\nabla_i\in {\rm H_{{\mathbb S }}^2}({\mathfrak n}_{3}).$ 
	The automorphism group of ${\mathfrak n}_{3}$ consists of invertible matrices of the form
\begin{center}	
$	\phi=
	\begin{pmatrix}
	x &  y  &  0 & 0\\
	z &  w  &  0 & 0\\
	q &  p  &  xw-yz & u\\
	r &  t  &  0 & v \\
	\end{pmatrix}.$ 
	
	\end{center}
	Since
	$$
	\phi_1^T\begin{pmatrix}
	\alpha_1 &  \alpha_2 & \alpha_4 & \alpha_6\\
	0  &  \alpha_3 & \alpha_5 &  \alpha_8 \\
	-\alpha_4 &  -\alpha_5 & 0    & 0\\
	 \alpha_7 &  \alpha_9 & 0    & \alpha_{10}\\
	\end{pmatrix} \phi_1=	\begin{pmatrix}
	\alpha^*_1 &  \alpha^*_2-\alpha^* & \alpha^*_4 & \alpha^*_6\\
	\alpha^*   &  \alpha^*_3 & \alpha^*_5 &  \alpha^*_8 \\
	-\alpha^*_4 &  -\alpha^*_5 & 0    & 0\\
	 \alpha^*_7 &  \alpha^*_9 & 0    & \alpha^*_{10}\\
	\end{pmatrix},
	$$
	 we have that the action of ${\rm Aut} ({\mathfrak n}_{3})$ on the subspace
$\langle \sum\limits_{i=1}^{10}\alpha_i\nabla_i  \rangle$
is given by
$\langle \sum\limits_{i=1}^{10}\alpha_i^{*}\nabla_i\rangle,$
where
\begin{longtable}{lll}
$\alpha^*_1$&$ =$&$   x^2 \alpha_1+r^2 \alpha_{10}+x z \alpha_2+z^2 \alpha_3+r x (\alpha_6+\alpha_7)+r z (\alpha_8+\alpha_9),$ \\
$\alpha^*_2$&$ =$&$   x (2 y \alpha_1+w \alpha_2+t (\alpha_6+\alpha_7))+z (y \alpha_2+2 w \alpha_3+t (\alpha_8+\alpha_9))+$\\&&$r (2 t \alpha_{10}+y (\alpha_6+\alpha_7)+w (\alpha_8+\alpha_9)),$ \\
$\alpha^*_3$&$ =$&$   y^2 \alpha_1+t^2 \alpha_{10}+w y \alpha_2+w^2 \alpha_3+t y (\alpha_6+\alpha_7)+t w (\alpha_8+\alpha_9),$ \\
$\alpha^*_4$&$ =$&$   (w x-y z) (x \alpha_4+z \alpha_5),$ \\
$\alpha^*_5$&$ =$&$ (w x-y z) (y \alpha_4+w \alpha_5),$\\
$\alpha^*_6$&$ =$&$ r v \alpha_{10}+u (x \alpha_4+ z \alpha_5)+v (x \alpha_6+ z \alpha_8),$\\
$\alpha^*_7$&$ =$&$ r v \alpha_{10}-u (x \alpha_4+z \alpha_5)+v (x \alpha_7+z \alpha_9),$\\
$\alpha^*_8$&$ =$&$ t v \alpha_{10}+u (y \alpha_4+ w \alpha_5)+v (y \alpha_6+ w \alpha_8),$\\
$\alpha^*_9$&$ =$&$ t v \alpha_{10}-u (y \alpha_4+w \alpha_5)+v (y \alpha_7+w \alpha_9),$\\
$\alpha^*_{10}$&$ =$&$ v^2 \alpha_{10}.$
\end{longtable}

Since, we are interested only in new algebras, we have  $ (\alpha_4, \alpha_5) \neq0$
and $(\alpha_6, \ldots, \alpha_{10})\neq 0.$
Then 

\begin{enumerate}
\item if $\alpha_{10}\neq 0,$ 
then we can suppose that $\alpha_{10}=1$ and $\alpha_4\neq 0, \alpha_5=0.$
By choosing
\begin{center}$y=0,$ 
$w=\frac{v^2}{x^2 \alpha_4},$
$r=-\frac{x (\alpha_6+\alpha_7)+z (\alpha_8+\alpha_9)}{2},$
$t=-\frac{v^2 \alpha_9}{x^2 \alpha_4},$
$u=\frac{v (x (\alpha_7-\alpha_6)+z (\alpha_9-\alpha_8))}{2 x \alpha_4},$
\end{center}
we have $\alpha_5^*=\alpha_6^*=\alpha_7^*=\alpha_9^*=0$ and 
\begin{longtable}{llll}
\multicolumn{4}{l}{$\alpha^*_1=\frac{4 x z \alpha_2+x^2 (4 \alpha_1-(\alpha_6+\alpha_7)^2)-2 x z (\alpha_6+\alpha_7) (\alpha_8+\alpha_9)+z^2 (4 \alpha_3-(\alpha_8+\alpha_9)^2)}{4},$}\\

\multicolumn{4}{l}{$\alpha^*_2=\frac{v^2 (2 x \alpha_2-x (\alpha_6+\alpha_7) (\alpha_8+\alpha_9)+z (4 \alpha_3-(\alpha_8+\alpha_9)^2))}{2 x^2 \alpha_4},$}\\

$\alpha^*_3=\frac{v^4 (\alpha_3-\alpha_8 \alpha_9)}{x^4 \alpha_4^2},$&

$\alpha^*_4=v^2,$&

$\alpha^*_8=\frac{v^3 (\alpha_8-\alpha_9)}{x^2 \alpha_4},$&

$\alpha^*_{10}=v^2.$
    \end{longtable}
Hence,
\begin{enumerate}
    \item\label{89first} if $\alpha_8\neq \alpha_9$ and $4\alpha_3\neq(\alpha_8+\alpha_9)^2,$
    then for $v=\frac{x^2 \alpha_4}{\alpha_8-\alpha_9}$ and $z= \frac{x (2 \alpha_2-(\alpha_6+\alpha_7) (\alpha_8+\alpha_9)) }{(\alpha_8+\alpha_9)^2-4 \alpha_3},$ we have two families of representatives
\begin{center}    
    $\langle 
    \nabla_1+\alpha \nabla_3+\nabla_4+\nabla_8+\nabla_{10}
    \rangle_{\alpha\neq \frac{1}{4}}$ 
and
    $\langle 
    \alpha \nabla_3+\nabla_4+\nabla_8+\nabla_{10}
    \rangle_{\alpha\neq \frac{1}{4}},$ 
\end{center}
    depending on 
    $\alpha_2^2+\alpha_3 (\alpha_6+\alpha_7)^2-\alpha_2 (\alpha_6+\alpha_7) (\alpha_8+\alpha_9)=\alpha_1 (4 \alpha_3-(\alpha_8+\alpha_9)^2)$ or not.
    
  \item if $\alpha_8\neq \alpha_9,$  $4\alpha_3=(\alpha_8+\alpha_9)^2$ and
  $2 \alpha_2\neq (\alpha_6+\alpha_7) (\alpha_8+\alpha_9),$
    then for  
\begin{center}    $x=\frac{2 \alpha_2-(\alpha_6+\alpha_7) (\alpha_8+\alpha_9)}{2 \alpha_4},$
    $v=\frac{((\alpha_6+\alpha_7) (\alpha_8+\alpha_9)-2 \alpha_2)^2}{4 \alpha_4 (\alpha_8-\alpha_9)}$ and 
    $z=\frac{(\alpha_6+\alpha_7)^2-4 \alpha_1}{4 \alpha_4},$  \end{center}
    we have   the representative
    $\langle 
    \nabla_2+\frac{1}{4} \nabla_3+\nabla_4+\nabla_8+\nabla_{10}
    \rangle.$ 
   \item if $\alpha_8\neq \alpha_9,$  $4\alpha_3=(\alpha_8+\alpha_9)^2$ and
  $2 \alpha_2= (\alpha_6+\alpha_7) (\alpha_8+\alpha_9),$
    then   for $v=\frac{x^2\alpha_4}{\alpha_8-\alpha_9}$,      we have
 \begin{center}    $\left\langle 
   \frac{x^2 (4 \alpha_1-(\alpha_6+\alpha_7)^2)}{4}  \nabla_1+\frac{x^4 \alpha_4^2}{4 (\alpha_8-\alpha_9)^2} \nabla_3+ \frac{x^4 \alpha_4^2}{(\alpha_8-\alpha_9)^2}(\nabla_4+  \nabla_8+ \nabla_{10})
    \right\rangle,$\end{center}
    which gives one from the following  two representatives
\begin{center}    
    $\langle 
    \nabla_1+\frac{1}{4} \nabla_3+\nabla_4+\nabla_8+\nabla_{10}
    \rangle$ 
and
    $\langle 
    \frac{1}{4}  \nabla_3+\nabla_4+\nabla_8+\nabla_{10}
    \rangle,$ 
\end{center}
    depending on 
    $4 \alpha_1=(\alpha_6+\alpha_7)^2$ or not. These orbits shall  be jointed with families which were found in the case (\ref{89first}).
    
    \item if $\alpha_8=\alpha_9$ and $\alpha_3\neq\alpha_9^2,$ 
    then for $v=\frac{\alpha_4 x^2}{\sqrt{\alpha_3-\alpha_9^2}}$ and 
    $z=\frac{x ((\alpha_6+\alpha_7) \alpha_9-\alpha_2)}{2 (\alpha_3-\alpha_9^2)},$ we have
    one from the following two representatives
 \begin{center}    
    $\langle 
    \nabla_1+ \nabla_3+\nabla_4+ \nabla_{10}
    \rangle$ 
and
    $\langle 
       \nabla_3+\nabla_4+ \nabla_{10}
    \rangle,$ 
\end{center}   
depending on $\alpha_2^2+\alpha_3 (\alpha_6+\alpha_7)^2=2 \alpha_2 (\alpha_6+\alpha_7) \alpha_9+4 \alpha_1 (\alpha_3-\alpha_9^2)$ or not.

\item  if $\alpha_8=\alpha_9$ and $\alpha_3=\alpha_9^2,$ then we have the representative
\begin{center}
$\left\langle 
   \frac{x (x (4 \alpha_1-(\alpha_6+\alpha_7)^2)+4 z (\alpha_2-(\alpha_6+\alpha_7) \alpha_9))}{4}\nabla_1+ \frac{v^2 (\alpha_2-(\alpha_6+\alpha_7) \alpha_9)}{x \alpha_4}\nabla_2+v^2\nabla_4+ v^2\nabla_{10}
    \right\rangle,$ 
\end{center}
which gives one of the following representatives 
\begin{center}
$\left\langle \nabla_2+ \nabla_4+ \nabla_{10} \right\rangle,$ 
$\left\langle \nabla_1+ \nabla_4+ \nabla_{10} \right\rangle$
or
$\left\langle \nabla_4+ \nabla_{10} \right\rangle,$ 
\end{center}
depending on $\alpha_2=(\alpha_6+\alpha_7) \alpha_9,$ $4 \alpha_1=(\alpha_6+\alpha_7)^2$ or not.

\end{enumerate}

\item if $\alpha_{10}=0$ and $\alpha_9\neq 0,$
 then we can suppose that $\alpha_9=1$ and $\alpha_4\neq 0, \alpha_5=0.$ 
 
 \begin{enumerate}
     \item\label{89sec}if $\alpha_8\neq -1,$ then choosing 
     $x=1,$ 
     $y=0,$ 
     $z=-\frac{ \alpha_6+\alpha_7}{1+\alpha_8},$
     $r=\frac{2  \alpha_3 (\alpha_6+\alpha_7)- \alpha_2 (1+\alpha_8)}{(1+\alpha_8)^2},$
     $t=-\frac{w \alpha_3}{1+\alpha_8},$
     $u=\frac{\alpha_7 \alpha_8-\alpha_6}{1+\alpha_8}$
     and $v= \alpha_4,$  we have two families of representatives 
\begin{center}    
    $\langle 
    \nabla_1+ \nabla_4+\alpha \nabla_8+ \nabla_9
    \rangle_{\alpha \neq-1}$ 
and
    $\langle 
  \nabla_4+\alpha \nabla_8+ \nabla_9
    \rangle_{\alpha \neq-1},$ 
\end{center}   
depending on 
$\alpha_1 (1 + \alpha_8)^2 +  \alpha_3 (\alpha_6 + \alpha_7)^2 = \alpha_2 (\alpha_6 + \alpha_7)(1 + \alpha_8)$ or not.

     \item if $\alpha_8=-1$ and $\alpha_6=-\alpha_7, \alpha_3\neq 0,$ 
     then by choosing 
\begin{center}     $y=0,$
     $z=-\frac{x \alpha_2}{2 \alpha_3},$
     $w=\frac{x^2 \alpha_4}{\alpha_3},$
     $u=\frac{x^2 (2 \alpha_3 \alpha_7-\alpha_2)}{2 \alpha_3}$
     and
     $v=x^2 \alpha_4,$ \end{center}
     we have one of the following representatives
 \begin{center}    
    $\langle 
    \nabla_1+ \nabla_3+\nabla_4- \nabla_8+ \nabla_9
    \rangle$ 
and
    $\langle 
     \nabla_3+\nabla_4- \nabla_8+ \nabla_9
    \rangle,$ 
\end{center}   
depending on 
$4 \alpha_3\alpha_1 =\alpha_2^2$   or not.        

 \item if $\alpha_8=-1$ and $\alpha_6=-\alpha_7, \alpha_3= 0,$ 
     then by choosing 
\begin{center}
    $y=0,$  $z=x(\alpha_4-\alpha_7)$, 
    $u=x (z+x \alpha_7)$ and 
    $v=x (z+x \alpha_7),$
\end{center}
we have one from the following representatives 
\begin{center}
     $\langle 
\nabla_4- \nabla_8+ \nabla_9
    \rangle,$    $\langle 
    \nabla_1+\nabla_4- \nabla_8+ \nabla_9
    \rangle$, 
     $\langle 
    \nabla_2+\nabla_4- \nabla_8+ \nabla_9
    \rangle,$  or $\langle 
    \nabla_1+\nabla_2+\nabla_4- \nabla_8+ \nabla_9
    \rangle,$

\end{center}
depending on $\alpha_1=0, \alpha_2=0$ or not.  The first two orbits will be joint with the families found in the case (\ref{89sec}).

 \item if $\alpha_8=-1$ and $\alpha_6\neq-\alpha_7,$ 
     then by choosing 
     $x=\frac{w}{\alpha_6+\alpha_7},$
     $y=0,$
     $z=0,$
     $r=-\frac{w \alpha_1}{(\alpha_6+\alpha_7)^2},$
     $t=-\frac{w \alpha_2}{\alpha_6+\alpha_7},$
     $u=-\frac{w^2 \alpha_6}{(\alpha_6+\alpha_7)^2},$
     and 
     $v=\frac{w^2 \alpha_4}{(\alpha_6+\alpha_7)^2},$
     we have one from the following representatives
 \begin{center}
     $\langle 
\nabla_3+\nabla_4+\nabla_7- \nabla_8+ \nabla_9
    \rangle$ or
     $\langle 
\nabla_4+\nabla_7- \nabla_8+ \nabla_9
    \rangle,$

\end{center}
depending on $\alpha_3=0$ or not.    
 \end{enumerate}

\item if $\alpha_{10} =\alpha_9= 0$ and $\alpha_8\neq 0,$
 then we can suppose that $\alpha_8=1$ and $\alpha_4\neq 0, \alpha_5=0.$ By choosing 
 $x=1,$
 $y=0,$
 $z=-\alpha_6-\alpha_7,$
 $r=-\alpha_2+2 \alpha_3 (\alpha_6+\alpha_7),$
 $t=-w \alpha_3,$
 $u=\alpha_7$
 and  $v=\alpha_4$, 
 we have one from the following representatives
  \begin{center}
     $\langle 
\nabla_1+\nabla_4+\nabla_8 \rangle$ or
     $\langle 
\nabla_4+\nabla_8     \rangle,$

\end{center}
depending on $\alpha_1+\alpha_3 (\alpha_6+\alpha_7)^2 =\alpha_2(\alpha_6+\alpha_7)$ or not.

\item if $\alpha_{10} =\alpha_9= \alpha_8= 0$ and $\alpha_7\neq 0,$ then we can suppose that $\alpha_4\neq0,\alpha_5=0.$ By choosing 
$y=0$ and $u=\frac{v \alpha_7}{\alpha_4}$ we have $\alpha_5^*=\alpha_7^*=\alpha_8^*=\alpha_9^*=\alpha_{10}^*=0,$ which gives the following case.

\item if $\alpha_{10} =\alpha_9= \alpha_8= \alpha_7=0$ and $\alpha_6\neq 0,$ then we can suppose that $\alpha_6=1$ and $\alpha_4\neq0,\alpha_5=0.$ By choosing 
$y=0,$ $z=0,$ $w=1,$ $r=-\frac{x \alpha_1}{\alpha_6},$ $t=-\frac{\alpha_2}{\alpha_6},$
$u=0$ and $v=\frac{x \alpha_4}{\alpha_6},$ we have one from the following representatives 

 \begin{center}
     $\langle 
\nabla_3+\nabla_4+\nabla_6 \rangle$ or
     $\langle 
\nabla_4+\nabla_6     \rangle,$

\end{center}
depending on $\alpha_3=0$ or not.    
 \end{enumerate}

Consider the orbit $\left\langle \nabla_1+\nabla_2+\nabla_4-\nabla_8+\nabla_9\right\rangle$. Notice that by chossing $x=1$, $y=0$, $z=-1$, $w=1$, $u=-1$ and $v=1$, we obtain the representaive $
    \left\langle \nabla_2+\nabla_4-\nabla_8+\nabla_9\right\rangle$.
    
 Summarizing, we obtain the following distinct orbits:
\begin{center}
$    \left\langle \nabla_1+\alpha\nabla_3+\nabla_4+\nabla_8+\nabla_{10}\right\rangle, 
    \left\langle \alpha\nabla_3+\nabla_4+\nabla_8+\nabla_{10}\right\rangle, 
    \left\langle \nabla_2+\frac{1}{4}\nabla_3+\nabla_4+\nabla_8+\nabla_{10}\right\rangle,  
    $
    \\
$   
    \left\langle \nabla_1+\nabla_3+\nabla_4+\nabla_{10}\right\rangle,  \left\langle \nabla_3+\nabla_4+\nabla_{10}\right\rangle,\left\langle \nabla_2+\nabla_4+\nabla_{10}\right\rangle, 
    \left\langle \nabla_1+\nabla_4+\nabla_{10}\right\rangle,     $ 
    \\
$ 
    \left\langle \nabla_4+\nabla_{10}\right\rangle, 
    \left\langle \nabla_1+\nabla_4+\alpha\nabla_8+\nabla_9\right\rangle,\left\langle \nabla_4+\alpha\nabla_8+\nabla_9\right\rangle,   \left\langle \nabla_1+\nabla_3+\nabla_4-\nabla_8+\nabla_9\right\rangle, 
    $ 
    \\
$ 
    \left\langle \nabla_3+\nabla_4-\nabla_8+\nabla_9\right\rangle, \left\langle \nabla_2+\nabla_4-\nabla_8+\nabla_9\right\rangle,     \left\langle \nabla_3+\nabla_4+\nabla_7-\nabla_8+\nabla_9\right\rangle,
    $ 
    \\
$ 
     \left\langle \nabla_4+\nabla_7-\nabla_8+\nabla_9\right\rangle, \left\langle \nabla_1+\nabla_4+\nabla_8\right\rangle,   \left\langle \nabla_4+\nabla_8\right\rangle,
    $ 
    $\left\langle \nabla_3+\nabla_4+\nabla_6\right\rangle, \left\langle \nabla_4+\nabla_6\right\rangle$
\end{center}

\subsubsection{Central extensions of ${\mathfrak n}_{4}$}
	Let us use the following notations:
\begin{longtable}{lll}
$\nabla_1 =[\Delta_{11}],$& $\nabla_2 = [\Delta_{12}],$& $ \nabla_3 = [\Delta_{22}],$\\
$\nabla_4 = [\Delta_{13}-\Delta_{31}],$& $\nabla_5 = [\Delta_{24}-\Delta_{42}].$ 
\end{longtable}
Take $\theta=\sum\limits_{i=1}^5\alpha_i\nabla_i\in {\rm H_{{\mathbb S }}^2}({\mathfrak n}_{4}).$ 
	The automorphism group of ${\mathfrak n}_{4}$ consists of invertible matrices of the form
\begin{center}	
$	\phi=
	\begin{pmatrix}
	x &  z  &  0 & 0\\
	0 &  y  &  0 & 0\\
	q &  r  &  xy & 0\\
	w &  t  &  -qy & xy^2 \\
	\end{pmatrix}$  
	
	\end{center}
	Since
	$$
	\phi^T\begin{pmatrix}
	\alpha_1 &  \alpha_2 & \alpha_4 & 0\\
	0  &  \alpha_3 & 0 &  \alpha_5 \\
	-\alpha_4&  0 & 0    & 0\\
	0&  -\alpha_5 & 0    & 0\\
	\end{pmatrix} \phi=	\begin{pmatrix}
	\alpha_1^* &  \alpha_2^*-\alpha^* & \alpha_4^* & 0\\
	\alpha^*  &  \alpha_3^* & 0 &  \alpha_5^* \\
	-\alpha_4^*&  0 & 0    & 0\\
	0&  -\alpha_5^* & 0    & 0\\
	\end{pmatrix},
	$$
	 we have that the action of ${\rm Aut} ({\mathfrak n}_{4})$ on the subspace
$\langle \sum\limits_{i=1}^5\alpha_i\nabla_i  \rangle$
is given by
$\langle \sum\limits_{i=1}^5\alpha_i^{*}\nabla_i\rangle,$
where
\begin{longtable}{lll}
$\alpha^*_1=  \alpha_{1}x^2,$ &
$\alpha^*_2= 2 \alpha_{1}xz+\alpha_{2}xy,$ &
$\alpha^*_3=   \alpha_{1} z^2+\alpha_{2} yz +\alpha_{3}y^2,$ \\
$\alpha_4^*=   \alpha_{4}x^2 y, $ &
$\alpha_5^*=  \alpha_{5} x y^3.$ 
\end{longtable}

Since, we are interested only in new algebras, we have  $ \alpha_5 \neq0.$
Then 

\begin{enumerate}
    \item $\alpha_4\neq0$ and $\alpha_1 \neq 0,$  
    by choosing 
    $x=\frac{\alpha_1^2 \alpha_5}{\alpha_4^3},$
    $y=\frac{\alpha_1}{\alpha_4}$
    and
    $z=-\frac{\alpha_2}{2 \alpha_4},$
    we have the family of representatives 
    $\langle \nabla_1+\alpha \nabla_3+\nabla_4+\nabla_5 \rangle.$
    
    \item $\alpha_4\neq0$ and $\alpha_1=0,$ by choosing $x=\frac{y^2 \alpha_5}{\alpha_4},$
    we have the representative
    \begin{center}
        $\left\langle \frac{y^3 \alpha_2 \alpha_5}{\alpha_4} \nabla_2+y (z \alpha_2+y \alpha_3) \nabla_3+ \frac{y^5 \alpha_5^2}{\alpha_4}( \nabla_4 + \nabla_5) \right\rangle,$
    \end{center}
    which gives the following representatives 
\begin{center}     
$\langle   \nabla_2+  \nabla_4 + \nabla_5 \rangle, \ 
\langle   \nabla_3+  \nabla_4 + \nabla_5 \rangle, \ 
\langle  \nabla_4 + \nabla_5 \rangle,$
    \end{center}
    depending on $\alpha_2=0, \alpha_3=0$ or not.
    
    \item  $\alpha_4=0$ and $\alpha_1\neq0,$ 
    by choosing 
    $x=\frac{y^3 \alpha_5}{\alpha_1}$
    and $x=\frac{y^3 \alpha_5}{\alpha_1},$
    we have the representative 
    \begin{center}
    $\left\langle   4 y^6 \alpha_5^2 \nabla_1-y^2 (\alpha_2^2-4 \alpha_1 \alpha_3) \nabla_3 +  4 y^6 \alpha_5^2\nabla_5 \right\rangle,$
    \end{center}
    which gives 
    \begin{center}
        $\left\langle    \nabla_1+ \nabla_3+ \nabla_5 \right\rangle$
        or
                $\left\langle    \nabla_1+ \nabla_5 \right\rangle,$
    \end{center}
        depending on $\alpha_2^2=4 \alpha_1 \alpha_3$ or not.
        
    \item $\alpha_4=0,$ $\alpha_1=0$ and $\alpha_2\neq 0,$
    by choosing $x=1,$
    $y=\sqrt{\alpha_2 \alpha_5^{-1}}$
    and $z=-\frac{\alpha_3}{\sqrt{\alpha_2  \alpha_5}},$ we have the representative
         $\left\langle    \nabla_2+ \nabla_5 \right\rangle.$
   
   \item $\alpha_4=0,$ $\alpha_1=0$ and $\alpha_2= 0,$ we have the representatives
     \begin{center}
        $\left\langle     \nabla_3+ \nabla_5 \right\rangle$
        or
                $\left\langle     \nabla_5 \right\rangle,$
    \end{center}
        depending on $\alpha_3=0$ or not.
    
    \end{enumerate}

Summarizing, we obtain the following distinct orbits:
\begin{center}
$    \left\langle \nabla_1+\alpha\nabla_3+\nabla_4+\nabla_5\right\rangle, 
    \left\langle \nabla_2+\nabla_4+\nabla_5\right\rangle, 
    \left\langle \nabla_3+\nabla_4+\nabla_5\right\rangle,   \left\langle \nabla_4+\nabla_5\right\rangle$,\\
   $  \left\langle \nabla_1+\nabla_3+\nabla_5\right\rangle, \left\langle \nabla_1+\nabla_5\right\rangle, \left\langle \nabla_2+\nabla_5\right\rangle,\left\langle \nabla_3+\nabla_5\right\rangle,\left\langle \nabla_5\right\rangle.
    $
   
\end{center}

\subsubsection{Classification theorem for $5$-dimensional algebras}\label{secteoB}
Now we are ready summarize all results related to the algebraic classification of complex $5$-dimensional nilpotent symmetric Leibniz   algebras.

\begin{theoremC}

Let ${\mathbb S}$ be a complex $5$-dimensional non-split nilpotent symmetric Leibniz algebra.
Then ${\mathbb S}$ is a $2$-step nilpotent algebra,
or it is isomorphic to one algebra from the following list:

\begin{center}
 
\begin{longtable}{llllllllll}
$\mathbb{S}_{01}$ &$:$& $ e_1e_1=e_4$ & $e_1 e_2 =e_3+e_5$ & $e_2 e_1 =-e_3$ \\&& $e_2e_2=e_5$ & $e_2e_3 =e_4$ & $e_3e_2=-e_4$\\
\hline
$\mathbb{S}_{02}$ &$:$& $ e_1e_1=e_4$ & $e_1 e_2 =e_3$ & $e_2 e_1 =-e_3$ \\&& $e_2e_2=e_5$ & $e_2e_3 =e_4$ & $e_3e_2=-e_4$\\
\hline
$\mathbb{S}_{03}$ &$:$& $ e_1e_1=e_4$ & $e_1 e_2 =e_3+e_5$ & $e_2 e_1 =-e_3$ & $e_2e_3 =e_4$ & $e_3e_2=-e_4$\\
\hline
$\mathbb{S}_{04}$ &$:$& $ e_1e_1=e_5$ & $e_1 e_2 =e_3$ & $e_2 e_1 =-e_3$ & $e_2e_3 =e_4$ & $e_3e_2=-e_4$\\
\hline
$\mathbb{S}_{05}$ &$:$&  $e_1 e_2 =e_3+e_5$ & $e_1e_3=e_5$ & $e_2 e_1 =-e_3$ & $e_2e_2=e_4+e_5$ \\&& $e_2e_3 =e_4$ & $e_3e_1=-e_5$ & $e_3e_2=-e_4$\\
\hline
$\mathbb{S}_{06}$ &$:$& $e_1e_1=e_5$ &  $e_1 e_2 =e_3$  & $e_2 e_1 =-e_3$ \\&& $e_2e_2=e_4+e_5$ & $e_2e_3 =e_4$ & $e_3e_2=-e_4$\\
\hline
$\mathbb{S}_{07}$ &$:$& $e_1e_1=e_5$ &  $e_1 e_2 =e_3$  & $e_2 e_1 =-e_3$ \\&& $e_2e_2=e_4$ & $e_2e_3 =e_4$ & $e_3e_2=-e_4$\\


 
\hline
$\mathbb{S}_{08}$ &$:$& $e_1e_1=e_5$ &  $e_1 e_2 =e_3+e_4$  & $e_1e_3=e_5$ & $e_2 e_1 =-e_3$ \\&& $e_2e_3 =e_4$ & $e_3e_1=-e_5$ & $e_3e_2=-e_4$\\
\hline
$\mathbb{S}_{09}$ &$:$& $e_1 e_2 =e_3+e_4$  & $e_1e_3=e_5$ & $e_2 e_1 =-e_3$ \\&& $e_2e_3 =e_4$ & $e_3e_1=-e_5$ & $e_3e_2=-e_4$\\
\hline
$\mathbb{S}_{10}^\alpha$ &$:$& $e_1e_1=e_5$ &  $e_1 e_2 =e_3+e_4+\alpha e_5$  & $e_2e_1 =-e_3$ \\&& $e_2e_2=e_5$ & $e_2e_3 =e_4$ & $e_3e_2=-e_4$\\


\hline
$\mathbb{S}_{11}$ &$:$& $e_1e_1=e_5$ &  $e_1 e_2 =e_3+e_4+e_5$  & $e_2 e_1 =-e_3$ & $e_2e_3 =e_4$ & $e_3e_2=-e_4$\\
\hline
$\mathbb{S}_{12}$ &$:$& $e_1 e_2 =e_3+e_4+e_5$  & $e_2 e_1 =-e_3$ & $e_2e_2=e_5$ & $e_2e_3 =e_4$ & $e_3e_2=-e_4$\\
\hline
$\mathbb{S}_{13}$ &$:$& $e_1 e_2 =e_3+e_4$  & $e_2 e_1 =-e_3$ & $e_2e_2=e_5$ & $e_2e_3 =e_4$ & $e_3e_2=-e_4$\\
\hline
$\mathbb{S}_{14}$ &$:$& $e_1 e_2 =e_3+e_5$  & $e_2 e_1 =-e_3$ & $e_2e_3 =e_4$ & $e_3e_2=-e_4$\\
\hline
$\mathbb{S}_{15}$ &$:$& $e_1e_1=e_5$ & $e_1 e_2 =e_3+e_4$  & $e_2 e_1 =-e_3$ & $e_2e_3 =e_4$ & $e_3e_2=-e_4$\\
\hline
$\mathbb{S}_{16}$ &$:$& $e_1e_1=e_5$ & $e_1 e_2 =e_3$  & $e_1e_3=e_5$ & $e_2 e_1 =-e_3$ \\&& $e_2e_2=e_5$ & $e_2e_3 =e_4$ & $e_3e_1=-e_5$ & $e_3e_2=-e_4$\\
\hline
$\mathbb{S}_{17}$ &$:$& $e_1e_1=e_5$ & $e_1 e_2 =e_3$  & $e_1e_3=e_5$ & $e_2 e_1 =-e_3$ \\&& $e_2e_3 =e_4$ & $e_3e_1=-e_5$ & $e_3e_2=-e_4$\\

\hline
$\mathbb{S}_{18}$ &$:$& $e_1 e_2 =e_3$  & $e_1e_3=e_5$ & $e_2 e_1 =-e_3$ & $e_2e_2=e_5$ \\&& $e_2e_3 =e_4$ & $e_3e_1=-e_5$ & $e_3e_2=-e_4$\\
\hline
$\mathbb{S}_{19}$ &$:$& $e_1 e_2 =e_3$  & $e_1e_3=e_5$ & $e_2 e_1 =-e_3$ \\&& $e_2e_3 =e_4$ & $e_3e_1=-e_5$ & $e_3e_2=-e_4$\\
\hline
$\mathbb{S}_{20}$ &$:$& $e_1e_1=e_5$ & $e_1 e_2 =e_3$  & $e_2 e_1 =-e_3$ \\&& $e_2e_2=e_5$ & $e_2e_3 =e_4$ & $e_3e_2=-e_4$\\

\hline
$\mathbb{S}_{21}^{\alpha,\beta}$ &$:$& $e_1e_1=\alpha e_5$ &  $e_1 e_2 =e_3+e_4+\beta e_5$  & $e_1e_3=e_5$ & $e_2 e_1 =-e_3$ & $e_2e_2=e_5$ \\&& $e_2e_3 =e_4$ & $e_3e_1=-e_5$ & $e_3e_2=-e_4$\\

\hline
$\mathbb{S}_{22}^{\alpha}$ &$:$& $e_1e_1=e_5$ &  $e_1 e_2 =e_3$  & $e_1e_3=e_5$ & $e_2 e_1 =-e_3$ & $e_2e_2=\alpha e_5$ \\&&  $e_2e_4=e_5$ & $e_3e_1=-e_5$ & $e_4e_4=e_5$\\
\hline
$\mathbb{S}_{23}^{\alpha}$ &$:$&   $e_1 e_2 =e_3$  & $e_1e_3=e_5$ & $e_2 e_1 =-e_3$ & $e_2e_2=\alpha e_5$ \\&& $e_2e_4=e_5$ &   $e_3e_1=-e_5$ & $e_4e_4=e_5$\\
\hline
$\mathbb{S}_{24}$ &$:$& $e_1 e_2 =e_3+e_5$  & $e_1e_3=e_5$ & $e_2 e_1 =-e_3$ & $e_2e_2=\frac{1}{4} e_5$ \\&& $e_2e_4=e_5$    &  $e_3e_1=-e_5$ & $e_4e_4=e_5$\\
\hline
$\mathbb{S}_{25}$ &$:$& $e_1e_1=e_5$ &  $e_1 e_2 =e_3$  & $e_1e_3=e_5$ & $e_2 e_1 =-e_3$ \\&& $e_2e_2=e_5$  & $e_3e_1=-e_5$ & $e_4e_4=e_5$\\
\hline
$\mathbb{S}_{26}$ &$:$& $e_1 e_2 =e_3$  & $e_1e_3=e_5$ & $e_2 e_1 =-e_3$ \\
&& $e_2e_2=e_5$ & $e_3e_1=-e_5$  & $e_4e_4=e_5$\\
\hline
$\mathbb{S}_{27}$ &$:$& $e_1 e_2 =e_3+e_5$  & $e_1e_3=e_5$ & $e_2 e_1 =-e_3$ & $e_3e_1=-e_5$ & $e_4e_4=e_5$\\
  \hline 
$\mathbb{S}_{28}$ &$:$& $e_1e_1=e_5$ &$e_1 e_2 =e_3$  & $e_1e_3=e_5$ \\&&
$e_2 e_1 =-e_3$ & $e_3e_1=-e_5$  & $e_4e_4=e_5$\\
\hline
$\mathbb{S}_{29}$ &$:$& $e_1 e_2 =e_3$  & $e_1e_3=e_5$ & $e_2 e_1 =-e_3$ & $e_3e_1=-e_5$ & $e_4e_4=e_5$\\
\hline
$\mathbb{S}_{30}^{\alpha}$ &$:$& $e_1e_1=e_5$ & $e_1 e_2 =e_3$  & $e_1e_3=e_5$ & $e_2 e_1 =-e_3$ \\&& $e_2e_4=\alpha e_5$  & $e_3e_1=-e_5$ & $e_4e_2=e_5$\\
\hline
$\mathbb{S}_{31}^{\alpha}$ &$:$& $e_1 e_2 =e_3$  & $e_1e_3=e_5$ & $e_2 e_1 =-e_3$ \\&&
$e_2e_4=\alpha e_5$ & $e_3e_1=-e_5$ & $e_4e_2=e_5$\\
\hline
$\mathbb{S}_{32}$ &$:$& $e_1e_1=e_5$ & $e_1 e_2 =e_3$  & $e_1e_3=e_5$ & $e_2 e_1 =-e_3$ \\&& $e_2e_2=e_5$ & $e_2e_4=-e_5$ & $e_3e_1=-e_5$ & $e_4e_2=e_5$\\
\hline
$\mathbb{S}_{33}$ &$:$& $e_1 e_2 =e_3$  & $e_1e_3=e_5$ & $e_2 e_1 =-e_3$ & $e_2e_2=e_5$ \\&& $e_2e_4=-e_5$  & $e_3e_1=-e_5$ & $e_4e_2=e_5$\\
\hline
$\mathbb{S}_{34}$ &$:$& $e_1 e_2 =e_3+e_5$  & $e_1e_3=e_5$ & $e_2 e_1 =-e_3$\\& & $e_2e_4=-e_5$ & $e_3e_1=-e_5$  & $e_4e_2=e_5$\\

\hline
$\mathbb{S}_{35}$ &$:$& $e_1 e_2 =e_3$  & $e_1e_3=e_5$ & $e_2 e_1 =-e_3$ & $e_2e_2=e_5$ \\&& $e_2e_4=-e_5$ & $e_3e_1=-e_5$ & $e_4e_1=e_5$ & $e_4e_2=e_5$\\
\hline
$\mathbb{S}_{36}$ &$:$& $e_1 e_2 =e_3$  & $e_1e_3=e_5$ & $e_2 e_1 =-e_3$ & $e_2e_4=-e_5$ \\&& $e_3e_1=-e_5$  & $e_4e_1=e_5$ & $e_4e_2=e_5$\\
\hline
$\mathbb{S}_{37}$ &$:$& $e_1e_1=e_5$ & $e_1 e_2 =e_3$  & $e_1e_3=e_5$ \\&&
$e_2 e_1 =-e_3$ & $e_2e_4=e_5$ & $e_3e_1=-e_5$ \\
\hline
$\mathbb{S}_{38}$ &$:$& $e_1 e_2 =e_3$  & $e_1e_3=e_5$ & $e_2 e_1 =-e_3$ & $e_2e_4=e_5$ & $e_3e_1=-e_5$ \\
\hline
$\mathbb{S}_{39}$ &$:$& $e_1 e_2 =e_3$  & $e_1e_3=e_5$ & $e_1e_4=e_5$\\& & $e_2 e_1 =-e_3$ & $e_2e_2=e_5$   & $e_3e_1=-e_5$ \\
\hline
$\mathbb{S}_{40}$ &$:$& $e_1 e_2 =e_3$  & $e_1e_3=e_5$ & $e_1e_4=e_5$ & $e_2 e_1 =-e_3$ & $e_3e_1=-e_5$ \\

\hline
$\mathbb{S}_{41}^{\alpha}$ &$:$& $e_1e_1=e_5$ & $e_1 e_2 =e_3$  & $e_1e_3=e_5$ & $e_2 e_1 =-e_3$ & $e_2e_2=\alpha e_5$ \\&& $e_2e_3=e_4$ & $e_2e_4=e_5$ &  $e_3e_1=-e_5$ & $e_3e_2=-e_4$ & $e_4e_2=-e_5$\\
\hline
$\mathbb{S}_{42}$ &$:$& $e_1 e_2 =e_3+e_5$  & $e_1e_3=e_5$ & $e_2 e_1 =-e_3$ & $e_2e_3=e_4$ \\&& $e_2e_4=e_5$ &   $e_3e_1=-e_5$ & $e_3e_2=-e_4$ & $e_4e_2=-e_5$\\
\hline
$\mathbb{S}_{43}$ &$:$& $e_1 e_2 =e_3$  & $e_1e_3=e_5$ & $e_2 e_1 =-e_3$ & $e_2e_2=e_5$ & $e_2e_3=e_4$ \\&& $e_2e_4=e_5$ &  $e_3e_1=-e_5$ & $e_3e_2=-e_4$ & $e_4e_2=-e_5$\\
\hline
$\mathbb{S}_{44}$ &$:$& $e_1 e_2 =e_3$  & $e_1e_3=e_5$ & $e_2 e_1 =-e_3$ & $e_2e_3=e_4$ \\&& $e_2e_4=e_5$  &  $e_3e_1=-e_5$ & $e_3e_2=-e_4$ & $e_4e_2=-e_5$\\
\hline
$\mathbb{S}_{45}$ &$:$& $e_1e_1=e_5$ & $e_1 e_2 =e_3$  & $e_2 e_1 =-e_3$ & $e_2e_2=e_5$ \\&& $e_2e_3=e_4$  & $e_2e_4=e_5$ &  $e_3e_2=-e_4$ & $e_4e_2=-e_5$\\
\hline
$\mathbb{S}_{46}$ &$:$& $e_1e_1=e_5$ & $e_1 e_2 =e_3$  & $e_2 e_1 =-e_3$ & $e_2e_3=e_4$ \\&& $e_2e_4=e_5$  &  $e_3e_2=-e_4$ & $e_4e_2=-e_5$\\
\hline
$\mathbb{S}_{47}$ &$:$& $e_1 e_2 =e_3+e_5$  & $e_2 e_1 =-e_3$ & $e_2e_3=e_4$ \\&&
$e_2e_4=e_5$ &  $e_3e_2=-e_4$ & $e_4e_2=-e_5$\\
\hline
$\mathbb{S}_{48}$ &$:$& $e_1 e_2 =e_3$  & $e_2 e_1 =-e_3$ & $e_2e_2=e_5$ & $e_2e_3=e_4$ \\&&
$e_2e_4=e_5$ &  $e_3e_2=-e_4$ & $e_4e_2=-e_5$\\
\hline
$\mathbb{S}_{49}$ &$:$& $e_1 e_2 =e_3$  & $e_2 e_1 =-e_3$ & $e_2e_3=e_4$ \\&&
$e_2e_4=e_5$ &  $e_3e_2=-e_4$ & $e_4e_2=-e_5$\\

    \end{longtable}
    \end{center}

\end{theoremC}

\section{The geometric classification of nilpotent  algebras}

\subsection{Definitions and notation}
Given an $n$-dimensional vector space $\mathbb V$, the set ${\rm Hom}(\mathbb V \otimes \mathbb V,\mathbb V) \cong \mathbb V^* \otimes \mathbb V^* \otimes \mathbb V$
is a vector space of dimension $n^3$. This space has the structure of the affine variety $\mathbb{C}^{n^3}.$ Indeed, let us fix a basis $e_1,\dots,e_n$ of $\mathbb V$. Then any $\mu\in {\rm Hom}(\mathbb V \otimes \mathbb V,\mathbb V)$ is determined by $n^3$ structure constants $c_{ij}^k\in\mathbb{C}$ such that
$\mu(e_i\otimes e_j)=\sum\limits_{k=1}^nc_{ij}^ke_k$. A subset of ${\rm Hom}(\mathbb V \otimes \mathbb V,\mathbb V)$ is {\it Zariski-closed} if it can be defined by a set of polynomial equations in the variables $c_{ij}^k$ ($1\le i,j,k\le n$).

Let $T$ be a set of polynomial identities.
The set of algebra structures on $\mathbb V$ satisfying polynomial identities from $T$ forms a Zariski-closed subset of the variety ${\rm Hom}(\mathbb V \otimes \mathbb V,\mathbb V)$. We denote this subset by $\mathbb{L}(T)$.
The general linear group ${\rm GL}(\mathbb V)$ acts on $\mathbb{L}(T)$ by conjugations:
$$ (g * \mu )(x\otimes y) = g\mu(g^{-1}x\otimes g^{-1}y)$$
for $x,y\in \mathbb V$, $\mu\in \mathbb{L}(T)\subset {\rm Hom}(\mathbb V \otimes\mathbb V, \mathbb V)$ and $g\in {\rm GL}(\mathbb V)$.
Thus, $\mathbb{L}(T)$ is decomposed into ${\rm GL}(\mathbb V)$-orbits that correspond to the isomorphism classes of algebras.
Let $O(\mu)$ denote the orbit of $\mu\in\mathbb{L}(T)$ under the action of ${\rm GL}(\mathbb V)$ and $\overline{O(\mu)}$ denote the Zariski closure of $O(\mu)$.

Let $\mathcal A$ and $\mathcal B$ be two $n$-dimensional algebras satisfying the identities from $T$, and let $\mu,\lambda \in \mathbb{L}(T)$ represent $\mathcal A$ and $\mathcal B$, respectively.
We say that $\mathcal A$ degenerates to $\mathcal B$ and write $\mathcal A\to \mathcal B$ if $\lambda\in\overline{O(\mu)}$.
Note that in this case we have $\overline{O(\lambda)}\subset\overline{O(\mu)}$. Hence, the definition of a degeneration does not depend on the choice of $\mu$ and $\lambda$. If $\mathcal A\not\cong \mathcal B$, then the assertion $\mathcal A\to \mathcal B$ is called a {\it proper degeneration}. We write $\mathcal A\not\to \mathcal B$ if $\lambda\not\in\overline{O(\mu)}$.

Let $\mathcal A$ be represented by $\mu\in\mathbb{L}(T)$. Then  $\mathcal A$ is  {\it rigid} in $\mathbb{L}(T)$ if $O(\mu)$ is an open subset of $\mathbb{L}(T)$.
 Recall that a subset of a variety is called irreducible if it cannot be represented as a union of two non-trivial closed subsets.
 A maximal irreducible closed subset of a variety is called an {\it irreducible component}.
It is well known that any affine variety can be represented as a finite union of its irreducible components in a unique way.
The algebra $\mathcal A$ is rigid in $\mathbb{L}(T)$ if and only if $\overline{O(\mu)}$ is an irreducible component of $\mathbb{L}(T)$.

Given the spaces $U$ and $W$, we write simply $U>W$ instead of $\dim \,U>\dim \,W$.



\subsection{Method of the description of  degenerations of algebras}

In the present work we use the methods applied to Lie algebras in \cite{BC99,GRH,GRH2,S90}.
First of all, if $\mathcal A\to \mathcal B$ and $\mathcal A\not\cong \mathcal B$, then $\mathfrak{Der}(\mathcal A)<\mathfrak{Der}(\mathcal B)$, where $\mathfrak{Der}(\mathcal A)$ is the Lie algebra of derivations of $\mathcal A$. We compute the dimensions of algebras of derivations and check the assertion $\mathcal A\to \mathcal B$ only for such $\mathcal A$ and $\mathcal B$ that $\mathfrak{Der}(\mathcal A)<\mathfrak{Der}(\mathcal B)$.


To prove degenerations, we construct families of matrices parametrized by $t$. Namely, let $\mathcal A$ and $\mathcal B$ be two algebras represented by the structures $\mu$ and $\lambda$ from $\mathbb{L}(T)$ respectively. Let $e_1,\dots, e_n$ be a basis of $\mathbb  V$ and $c_{ij}^k$ ($1\le i,j,k\le n$) be the structure constants of $\lambda$ in this basis. If there exist $a_i^j(t)\in\mathbb{C}$ ($1\le i,j\le n$, $t\in\mathbb{C}^*$) such that $E_i^t=\sum\limits_{j=1}^na_i^j(t)e_j$ ($1\le i\le n$) form a basis of $\mathbb V$ for any $t\in\mathbb{C}^*$, and the structure constants of $\mu$ in the basis $E_1^t,\dots, E_n^t$ are such rational functions $c_{ij}^k(t)\in\mathbb{C}[t]$ that $c_{ij}^k(0)=c_{ij}^k$, then $\mathcal A\to \mathcal B$.
In this case  $E_1^t,\dots, E_n^t$ is called a {\it parametrized basis} for $\mathcal A\to \mathcal B$.
To simplify our equations, we will use the notation $A_i=\langle e_i,\dots,e_n\rangle,\ i=1,\ldots,n$ and write simply $A_pA_q\subset A_r$ instead of $c_{ij}^k=0$ ($i\geq p$, $j\geq q$, $k< r$).

Since the variety of $4$-dimensional nilpotent weakly associative (and $5$-dimensional nilpotent symmetric Leibniz) algebras  contains infinitely many non-isomorphic algebras, we have to do some additional work.
Let $\mathcal A(*):=\{\mathcal A(\alpha)\}_{\alpha\in I}$ be a series of algebras, and let $\mathcal B$ be another algebra. Suppose that for $\alpha\in I$, $\mathcal A(\alpha)$ is represented by the structure $\mu(\alpha)\in\mathbb{L}(T)$ and $B\in\mathbb{L}(T)$ is represented by the structure $\lambda$. Then we say that $\mathcal A(*)\to \mathcal B$ if $\lambda\in\overline{\{O(\mu(\alpha))\}_{\alpha\in I}}$, and $\mathcal A(*)\not\to \mathcal B$ if $\lambda\not\in\overline{\{O(\mu(\alpha))\}_{\alpha\in I}}$.

Let $\mathcal A(*)$, $\mathcal B$, $\mu(\alpha)$ ($\alpha\in I$) and $\lambda$ be as above. To prove $\mathcal A(*)\to \mathcal B$ it is enough to construct a family of pairs $(f(t), g(t))$ parametrized by $t\in\mathbb{C}^*$, where $f(t)\in I$ and $g(t)\in {\rm GL}(\mathbb V)$. Namely, let $e_1,\dots, e_n$ be a basis of $\mathbb V$ and $c_{ij}^k$ ($1\le i,j,k\le n$) be the structure constants of $\lambda$ in this basis. If we construct $a_i^j:\mathbb{C}^*\to \mathbb{C}$ ($1\le i,j\le n$) and $f: \mathbb{C}^* \to I$ such that $E_i^t=\sum\limits_{j=1}^na_i^j(t)e_j$ ($1\le i\le n$) form a basis of $\mathbb V$ for any  $t\in\mathbb{C}^*$, and the structure constants of $\mu_{f(t)}$ in the basis $E_1^t,\dots, E_n^t$ are such rational functions $c_{ij}^k(t)\in\mathbb{C}[t]$ that $c_{ij}^k(0)=c_{ij}^k$, then $\mathcal A(*)\to \mathcal B$. In this case  $E_1^t,\dots, E_n^t$ and $f(t)$ are called a parametrized basis and a {\it parametrized index} for $\mathcal A(*)\to \mathcal B$, respectively.

We now explain how to prove $\mathcal A(*)\not\to\mathcal  B$.
Note that if $\mathfrak{Der} \ \mathcal A(\alpha)  > \mathfrak{Der} \  \mathcal B$ for all $\alpha\in I$ then $\mathcal A(*)\not\to\mathcal B$.
One can also use the following  Lemma, whose proof is the same as the proof of Lemma 1.5 from \cite{GRH}.

\begin{lemma}\label{gmain}
Let $\mathfrak{B}$ be a Borel subgroup of ${\rm GL}(\mathbb V)$ and $\mathcal{R}\subset \mathbb{L}(T)$ be a $\mathfrak{B}$-stable closed subset.
If $\mathcal A(*) \to \mathcal B$ and for any $\alpha\in I$ the algebra $\mathcal A(\alpha)$ can be represented by a structure $\mu(\alpha)\in\mathcal{R}$, then there is $\lambda\in \mathcal{R}$ representing $\mathcal B$.
\end{lemma}

\subsection{The geometric classification of $4$-dimensional 
 nilpotent weakly associative and symmetric Leibniz algebras }
The main result of the present section is the following theorem.

\begin{theoremD}\label{geobl}
The variety of $4$-dimensional nilpotent symmetric Leibniz algebras has 
dimension {\it 11 }  and it has 
one rigid algebra $\mathcal{S}_{01}$ and three  irreducible components
defined by 
\begin{center}$\mathcal{C}_1=\overline{\{\mathcal{O}(\mathfrak{N}_2(\alpha))\}}$, $\mathcal{C}_2=\overline{\{\mathcal{O}(\mathfrak{N}_3(\alpha))\}}$ and $\mathcal{C}_3=\overline{\mathcal{O}(\mathcal{S}_{01})}$.
\end{center}
\medskip

The variety of $4$-dimensional nilpotent weakly associative   algebras  has 
dimension $16$  and it has
one rigid algebra $\mathcal{S}_{01}$ and 
$3$ irreducible components
defined by  
\begin{center}$\mathcal{C}_1=\overline{\{\mathcal{O}( \mathcal{C}_{19}^{\alpha})\}}$, 
$\mathcal{C}_2=\overline{\mathcal{O}({\mathcal W}_{06}^{\alpha})}$ \  and 
$\mathcal{C}_3=\overline{\mathcal{O}(\mathcal{S}_{01})}.$ 
\end{center}
\end{theoremD}

\begin{Proof}
The description of all irreducible components  of $4$-dimensional Zinbiel and nilpotent Leibniz algebras was given in \cite{kppv}.
Using the cited result, and noting that $\mathcal{S}_{01}\simeq\mathcal{L}_{11}$, $\mathcal{S}_{02}\simeq\mathcal{L}_{10}$, $\mathcal{S}_{03}\simeq\mathcal{L}_{9}$, $\mathcal{S}_{04}\simeq\mathcal{L}_{12}$ and $\mathcal{S}_{05}\simeq\mathcal{L}_{1}$, we can see that the variety of $4$-dimensional symmetric Leibniz algebras has three irreducible components given by the following two families of algebras:

\begin{longtable}{lllllllll}
${\mathfrak N}_2(\alpha)$  & $:$ & $e_1e_1 = e_3$ & $e_1e_2 = e_4$ & $e_2e_1 = -\alpha e_3$ & $e_2e_2 = -e_4$ \\
\hline
${\mathfrak N}_3(\alpha)$ & $:$ & $e_1e_1 = e_4$ & $e_1e_2 = \alpha e_4$ & $e_2e_1 = -\alpha e_4$ & $e_2e_2 = e_4$ & $e_3e_3=e_4$
\end{longtable}
and the rigid algebra
\begin{longtable}{lllllllll}
${\mathcal S}_{01}$  & $:$ & $e_1e_1 = e_4$ & $e_1e_2 = e_3$ & $e_2e_1 = -e_3$ & $e_2e_2 = e_4$ & $e_2e_3 = e_4$ & $e_3e_2 = -e_4$ \\
\end{longtable}

Now we can prove that the variety of $4$-dimensional nilpotent weakly associative algebras has $4$ irreducible components (for the algebraic classification, see Theorem A). The list of all necessary degenerations is given below:

\begin{longtable}{|lcl|ll|}
 
\hline

${\mathcal W}^{ \frac{1}{4\alpha^2}}_{06}$&$\to$&$\mathfrak{N}_{3}(\alpha)$ & 
$E_1^t=te_1$ & $E_2^t=\frac{t}{2\alpha}e_2$ \\
&&&$E_3^t=\frac{t}{2\alpha}e_3+\frac{t}{4\alpha}e_4$& $ E_4^t=\frac{t^2}{4\alpha^2}e_4$ 
\\

\hline
${\mathcal W}_{04} $&$\to$&$\mathcal{W}_{01}$ & 
$E_1^t=\sqrt{t}e_1+e_2$ & $E_2^t=te_2+e_4$ \\
&&&$E_3^t=\sqrt{t}e_3$& $ E_4^t=te_4$ \\
\hline

${\mathcal W}_{01}$&$\to$&$\mathcal{W}_{02}$ & 
$E_1^t=te_1$ & $E_2^t=t^2e_2$ \\
&&&$E_3^t=t^{2}e_3$& $ E_4^t=t^3e_4$ \\
\hline

${\mathcal W}_{04}$&$\to$&$\mathcal{W}_{03}^{\alpha}$ & 
$E_1^t=\sqrt{t}e_1+\alpha \sqrt{t}e_3$ & $E_2^t=te_2+t\alpha(\alpha+1)e_4$ \\
&&&$E_3^t=\sqrt{t^3}e_3$& $ E_4^t=t^2e_4$ \\
\hline

${\mathcal W}^{\frac{1-t}{t^2}}_{06}$&$\to$&$\mathcal{W}_{04}$ & 
$E_1^t=\frac{t}{2}e_1+\frac{1}{2}e_2$ & $E_2^t=\frac{t}{2}e_3+\frac{1}{2}e_4$ \\
&&&$E_3^t=\frac{t}{2}e_2$& $ E_4^t=\frac{t^2}{4}e_4$ 
\\
 
\hline
${\mathcal W}_{04}$&$\to$&$\mathcal{W}_{05}$ & 
$E_1^t=\sqrt{t^2-1}e_1$ & $E_2^t=(t^2-1)e_2$ \\
&&&$E_3^t=\frac{t^2-1}{t}e_2+\frac{\sqrt{(t^2-1)^3}}{t}e_3$& $ E_4^t=\frac{(t^2-1)^2}{t}e_4$ \\
\hline

${\mathcal W}^{t^2}_{06}$&$\to$&$\mathcal{W}_{07}$ & 
$E_1^t=t^{-1} e_1$ & $E_2^t=te_2$ \\
&&&$E_3^t=e_3$& $ E_4^t=e_4$ 
\\ 
\hline


  ${\mathcal W}_{06}^{-t^{-2}} $  &$\to$&$\mathcal{W}_{08}$ & 
$E_1^t=te_1+e_3$ & $E_2^t=te_2$ \\
&&&$E_3^t=t^2 e_3$& $ E_4^t=t^2 e_4$   \\
\hline

${\mathcal W}^{-t^{-2}}_{06}$&$\to$&$\mathcal{W}_{09}$ & 
$E_1^t=ie_2-e_3$ & $E_2^t=te_1-e_3$ \\
&&&$E_3^t=ite_3+(1+it)e_4$& $ E_4^t=-ite_4$ 
\\ 
 \hline

${\mathcal W}_{09}$&$\to$&$\mathcal{W}_{10}$ & 
$E_1^t=e_1$ & $E_2^t=te_2$ \\
&&&$E_3^t=te_3$& $ E_4^t=te_4$ \\
\hline
 
${\mathcal W}_{07}$&$\to$&$\mathcal{W}_{11}$ & 
$E_1^t=te_1$ & $E_2^t=t^{-1}e_2$\\
&&&$E_3^t=e_3$& $ E_4^t=e_4$ \\

 \hline



 ${\mathcal W}_{06}^
 {\frac{1 - 2 T}{4 T^2}}$&$\to$&$\mathfrak{N}_{2}(\alpha)$ & 
$E_1^t= \frac{t}{T-1}e_1 + \frac{2T-1}{2T(T-1)}e_2$ &
$E_2^t= -\frac{Tt}{T-1}e_1 - \frac{t}{2(T-1)}e_2$\\
\multicolumn{3}{|l|}{$T=\alpha + \sqrt{(\alpha-1) \alpha}$
}&
$E_3^t= -\frac{t^2}{(\alpha-1)T}e_3 + \frac{(2T-1)t^2}{2(\alpha-1)T^2}e_4$&

$E_4^t= -\frac{\alpha t^2}{(\alpha-1)T}e_3 - \frac{\alpha t^2}{2(\alpha-1)T^2}e_4$\\

 \hline

\end{longtable}

Recall that the variety of $4$-dimensional nilpotent commutative algebras has one irreducible component of dimension $16$ (see, \cite{fkkv}), defined by the following family of algebras:
\begin{longtable}{llllllllllllllllll}
 $ \mathcal{C}_{19}^{\alpha} $ & $:$&   $  e_1 e_1 = e_2$ &$  e_1e_3=\alpha e_4$  &$  e_2 e_2=e_3$  &$ e_2e_3= e_4$  &$  e_3e_3=e_4$  
\end{longtable}
It is easy to see that   $\dim \mathcal{O}({\mathcal W}_{06}^{\alpha})=15$ and it gives an irreducible component because all other nilpotent non-commutative weakly associative algebras have a smaller dimension of orbit closure.
Hence, the commutative family $\mathcal{C}_{19}^{\alpha}$ gives the second irreducible component in the variety of nilpotent weakly associative algebras.
The Algebra   ${\mathcal S}_{01}$  is not in the orbit closure of  ${\mathcal W}_{06}^{\alpha}$ since it does not satisfy the following 
condition $\{ A_1^2 \subseteq A_3, \ A_1^3 \subseteq A_4, \ c_{12}^3=c_{21}^3 \}.$ 
Hence, it gives a  irreducible component.




\end{Proof}

\subsection{The geometric classification of $4$-dimensional 
 symmetric Leibniz algebras 
 and conjectures about nilpotent algebras}
The main result of the present section is the following theorem and the corollary from it.
\begin{theoremE}\label{geo2}
The variety of $4$-dimensional  symmetric Leibniz     algebras  has   
dimension $13$  and it has 
one rigid algebra ${\mathfrak L}_{02}$ and
five irreducible components
defined by  

\begin{center}
$\mathcal{C}_1=\overline{\{\mathcal{O}(\mathfrak{N}_2(\alpha))\}},$ \ $\mathcal{C}_2=\overline{\{\mathcal{O}(\mathfrak{N}_3(\alpha))\}},$ \
$\mathcal{C}_3=\overline{ \mathcal{O}(\mathfrak{L}_{02}) },$ \  $\mathcal{C}_4=\overline{\{\mathcal{O}(\mathfrak{L}_{15}^\alpha)\}}$   and 
$\mathcal{C}_5=\overline{\{\mathcal{O}(\mathfrak{L}_{24}^\alpha)\}}.$  

\end{center}

 \end{theoremE}

\begin{Proof}

Thanks to \cite{ikv17}, we have that 
${\mathcal S}_{01},{\mathfrak L}_{14},{\mathfrak L}_{16},{\mathfrak L}_{36}$
can not give irreducible components in the variety of symmetric Leibniz algebras.
All needed degenerations are given in the table below:
\begin{center}  
\begin{longtable}{|lll|llll|}
\hline
${\mathfrak L}^{-\frac{\alpha}{(\alpha-1)^2}}_{15}$&$\to$&$\mathfrak{L}_{35}^{\alpha}$ & 
$E_1^t=  e_1 + \frac{1}{\alpha-1} e_2$ & $ E_2^t=t e_2$ & $E_3^t=\frac{t}{\alpha-1}e_4 $ &$E_4^t=e_3$
\\ 
 \hline

${\mathfrak L}^{-t^{-1}}_{35}$&$\to$&$\mathfrak{L}_{27}$ & 
$E_1^t=-e_1$ & $ E_2^t=e_4$ & $E_3^t=te_2$ &$E_4^t=e_3$
\\ 
 \hline

${\mathfrak L}^{1+t}_{24}$&$\to$&$\mathfrak{L}_{30}$ & 
$E_1^t=e_1$ & $E_2^t=e_2+e_3$ &$E_3^t=te_3$& $ E_4^t=te_4$ 
\\ 
 \hline
 \end{longtable}
\end{center}
Algebras $\mathfrak{L}_{15}^{\alpha}$ give an irreducible component
in all (one-sided) Leibniz algebras \cite{ikv17}. Hence, it gives irreducible components in symmetric Leibniz algebras.
$\mathfrak{L}_{15}^{\alpha} \not\to \mathfrak{L}_{2},\mathfrak{L}_{24}^{\alpha}$ by $S$-invariant from \cite[Proof of Theorem 2]{ikv17}. 
Algebras $\mathfrak{L}_{2}$ and $\mathfrak{L}_{24}^{\alpha}$ have $12$-dimensional geometrical varieties, and from here, 
we have two new irreducible components.

Algebras $\mathfrak{N}_2(\alpha)$ and $\mathfrak{N}_3(\alpha)$ give two irreducible components by the following reasons:

\begin{enumerate}
    \item ${\mathfrak L}^{\alpha}_{24}$ and ${\mathfrak L}_{2}$ have a $3$-dimensional anticommutative subalgebra, but $\mathfrak{N}_2(\alpha)$ and $\mathfrak{N}_3(\alpha)$ not.
    
    \item ${\mathfrak L}^{\alpha}_{15}$ satisfies the following conditions
    $\{  A_1^2 \subseteq A_3, c_{11}^3=c_{22}^3=0\},$  but $\mathfrak{N}_2(\alpha)$ and $\mathfrak{N}_3(\alpha)$ not.
    
\end{enumerate}

\end{Proof}

\subsubsection{Conjectures about nilpotent algebras}
Several conjectures state that nilpotent Lie algebras form a very small subvariety in the variety of Lie algebras. Grunewald and O'Halloran conjectured in \cite{GRH3} that for any $n$-dimensional nilpotent Lie algebra $A$ there exists an $n$-dimensional non-nilpotent Lie algebra $B$ such that $B\to A$. At the same time,
Vergne conjectured in \cite{V70} that a nilpotent Lie algebra cannot be rigid in the variety of all Lie algebras. Analogous assertions can be conjectured for other varieties.
We will say that the variety $\mathfrak{A}$ of algebras has {\it Grunewald--O'Halloran Property} if for any nilpotent algebra $A\in\mathfrak{A}$ there is a non-nilpotent algebra $B\in\mathfrak{A}$ such that $B\to A$.
We will say that $\mathfrak{A}$ has {\it Vergne Property} if there are no nilpotent rigid algebras in $\mathfrak{A}$.
We will also say that $\mathfrak{A}$ has {\it Vergne--Grunewald--O'Halloran Property} if any irreducible component of $\mathfrak{A}$ contains a non-nilpotent algebra.
Grunewald--O'Halloran Property was proved for four dimensional Lie  algebras in \cite{BC99}
and for three dimensional  Leibniz algebras in \cite{CKLO13}. Also, some results concerning Grunewald--O'Halloran Conjecture for Lie algebras were obtained in \cite{tirao,ht16}.
It is clear that Vergne--Grunewald--O'Halloran Property follows from Grunewald--O'Halloran Property and Vergne Property follows from Vergne--Grunewald--O'Halloran Property.
It was proven that the variety of complex $4$-dimensional (one-sided) Leibniz algebras has Vergne--Grunewald--O'Halloran Property and on the other hand, it does not have Grunewald--O'Halloran Property \cite{ikv17}.
As a corollary of the Theorem E, we have

\begin{corollaryA}
The variety of complex $4$-dimensional symmetric Leibniz algebras has no Vergne--Grunewald--O'Halloran Property. However, it has  Vergne Property.
\end{corollaryA}

\subsection{The geometric classification of $5$-dimensional nilpotent 
 symmetric Leibniz algebras}
The main result of the present section is the following theorem.

\begin{theoremF}\label{geo3}
The variety of $5$-dimensional nilpotent symmetric Leibniz     algebras  has 
dimension  $24$   and it has 
$6$  irreducible components
defined by  
 
\begin{center}
$\mathcal{C}_1=\overline{\{\mathcal{O}({\mathfrak V}_{4+1})\}},$ \ $\mathcal{C}_2=\overline{\{\mathcal{O}({\mathfrak V}_{3+2})\}},$ \
$\mathcal{C}_3=\overline{\{\mathcal{O}({\mathfrak V}_{2+3})\}},$ \\
$\mathcal{C}_4=\overline{\{\mathcal{O}(\mathbb{S}_{21}^{\alpha,\beta})\}},$ \  
$\mathcal{C}_5=\overline{\{\mathcal{O}(\mathbb{S}_{22}^{\alpha})\}}$ \ and  
$\mathcal{C}_6=\overline{\{\mathcal{O}(\mathbb{S}_{41}^{\alpha})\}}.$   

\end{center}

In particular, there are no rigid algebras in this variety.
 
\end{theoremF}

\begin{proof}
Thanks to \cite{ikp20} the variety of $5$-dimensional $2$-step nilpotent algebras has only three irreducible components defined by 

\begin{longtable}{lllllll}
${\mathfrak V}_{4+1}$ & $:$&  
$e_1e_2=e_5$& $e_2e_1=\lambda e_5$ &$e_3e_4=e_5$&$e_4e_3=\mu e_5$\\

${\mathfrak V}_{3+2}$ &$ :$&
$e_1e_1 =  e_4$& $e_1e_2 = \mu_1 e_5$ & $e_1e_3 =\mu_2 e_5$& 
$e_2e_1 = \mu_3 e_5$  & $e_2e_2 = \mu_4 e_5$  \\
& & $e_2e_3 = \mu_5 e_5$  & $e_3e_1 = \mu_6 e_5$  & \multicolumn{2}{l}{$e_3e_2 = \lambda e_4+ \mu_7 e_5$ } & $e_3e_3 =  e_5$  \\

${\mathfrak V}_{2+3}$ &$ :$&
$e_1e_1 = e_3 + \lambda e_5$& $e_1e_2 = e_3$ & $e_2e_1 = e_4$& $e_2e_2 = e_5$

\end{longtable}

Thanks to Theorem D, 
all $5$-dimensional split symmetric Leibniz  algebras are in orbit closure 
of families ${\mathfrak V}_{4+1}$ and ${\mathfrak V}_{3+2},$
and 
$\mathcal{S}_{01}$ (considered as a $5$-dimensional algebra).  
After a careful  checking  dimensions of orbit closures of the more important for us algebras, we have 

\begin{center}  
$\dim  \mathcal{O}({\mathfrak V}_{3+2})=24,$ \, 
$
\dim \mathcal{O}(\mathbb{S}_{21}^{\alpha,\beta})=\dim \mathcal{O}(\mathbb{S}_{22}^{\alpha}) = 21,$

$\dim \mathcal{O}(\mathbb{S}_{41}^{\alpha})=
\dim \mathcal{O}({\mathfrak V}_{4+1})=20, \,
\dim \mathcal{O}({\mathfrak V}_{2+3})=18.$

 \end{center}
 Hence, 
$ \mathbb{S}_{21}^{\alpha,\beta}, \mathbb{S}_{22}^{\alpha},$ and ${\mathfrak V}_{3+2}$ give $3$ irreducible components.
Below we have listed all important reasons for necessary non-degenerations.


\begin{longtable}{|lcl|l|}
\hline
\multicolumn{4}{|c|}{\textrm{
{\bf  Non-degenerations reasons}}
}  \\
\hline






 




$  \begin{array}{l}
\mathbb{S}_{21}^{\alpha,\beta} \\ 
\end{array}  $  &$\not \to$& 
$\begin{array}{l}
\mathbb{S}_{41}^{\alpha},\\ 
{\mathfrak V}_{4+1},\\ 
{\mathfrak V}_{2+3}\\ 
\end{array} $
& ${\mathcal R}=
\left\{
\begin{array}{l}
A_1^2 \subseteq A_3, \,
A_1A_4+A_4A_1=0, \\ 

c_{11}^3=0,\, 
c_{22}^3=c_{22}^4=0,\, 


c_{12}^3=-c_{21}^3, \,

 

\end{array}
\right\}$
 \\
\hline

$\begin{array}{l}
\mathbb{S}_{22}^{\alpha}
\end{array}  $  &$\not \to$& 
$\begin{array}{l}
\mathbb{S}_{41}^{\alpha},\\
{\mathfrak V}_{2+3},\\  
  {\mathfrak V}_{4+1}  
\end{array} $
& ${\mathcal R}=
\left\{
\begin{array}{l}
A_1^2  \subseteq A_3, \,
A_1A_3+A_3A_1+A_2^2 \subseteq A_5, \,
A_1A_5+A_5A_1=0, \\ 

c_{11}^3=c_{11}^4=0,\, 
 
c_{12}^3=-c_{21}^3, \,
c_{12}^4=-c_{21}^4, \, 
 c_{34}^5 = c_{43}^5, \\ 

c_{24}^5 c_{33}^5= c_{23}^5 c_{43}^5, \,
c_{32}^5 c_{43}^5= c_{33}^5 c_{42}^5,\,
c_{23}^5 c_{42}^5= c_{24}^5 c_{32}^5, \\ 
c_{44}^5 c_{23}^5=  c_{24}^5 c_{34}^5, \,
c_{44}^5 c_{33}^5=  c_{34}^5 c_{43}^5, \,
c_{44}^5 c_{32}^5=  c_{42}^5 c_{34}^5, \\

c_{13}^5 c_{43}^5 +c_{31}^5 c_{43}^5=c_{14}^5 c_{33}^5 + c_{33}^5 c_{41}^5, \\

c_{14}^5 c_{43}^5  +c_{41}^5 c_{43}^5= c_{31}^5 c_{44}^5 + c_{13}^5 c_{44}^5,\\
c_{41}^5 c_{23}^5 + c_{14}^5 c_{23}^5 = c_{13}^5 c_{24}^5+c_{24}^5c_{31}^5

\end{array}
\right\}$
 \\
\hline

$\begin{array}{l}
\mathbb{S}_{41}^{\alpha} \\ 
\end{array}  $  &$\not \to$& 
$\begin{array}{l}
{\mathfrak V}_{2+3}\\ 
\end{array} $
& ${\mathcal R}=
\left\{
\begin{array}{l}
A_1^2  \subseteq A_3, \,

c_{11}^3=c_{11}^4=0,\, 
c_{22}^3=c_{22}^4=0,\,


 

\end{array}
\right\}$
 \\
\hline

\end{longtable}







The rest of the degenerations is given below on the following two tables and it completes the proof of the Theorem.

\begin{longtable}{|lclll|}
 
\hline

$\mathbb{S}^{2,1}_{21}$&$\to$&$\mathcal{S}_{01}$ & 
$E_1^t=e_1-e_2$ & $E_2^t=e_2$\\
\multicolumn{3}{|l}{$E_3^t=e_3+e_5$} & $E_4^t=e_5$ & $E_5^t=-t^{-1}e_4+t^{-1}e_5$\\ \hline

$\mathbb{S}_{21}^{-t, 2i \sqrt{t}}$&$\to$&$\mathbb{S}_{01}$ & 
$E_1^t=te_2$ & $E_2^t=i \sqrt{t} e_1 + t e_2 + i \sqrt{t} e_3 $ \\ 
\multicolumn{3}{|l}{$E_3^t=-i \sqrt{t^3} e_3 + t^2  e_5$} & $ E_4^t=t^2 e_5$ & $E_5^t=i \sqrt{t^3} e_4$ \\ \hline

$\mathbb{S}_{21}^{2+t, -2}$&$\to$&$\mathbb{S}_{05}$ & 
$E_1^t=-e_2$ & $E_2^t=t^{-1}e_1+t^{-1}e_2+2t^{-1}e_3$ \\ 
\multicolumn{3}{|l}{$E_3^t=t^{-1}e_3$} & $ E_4^t=t^{-2}e_4+t^{-2}e_5$ & $E_5^t=t^{-1}e_5$ \\ \hline

$\mathbb{S}_{21}^{-t^2, it^{-1}+2it}$&$\to$&$\mathbb{S}_{06}$ & 
$E_1^t=i t e_1 + t^2 e_2$ & $E_2^t= t e_1 + i t^2 e_2 + (t + 2 t^3) e_3$ \\ 
\multicolumn{3}{|l}{$E_3^t=-2 t^3 e_3 + 2 t^5 e_4$} & $ E_4^t= -4 t^4 e_5$ & $E_5^t= i t^3 e_4 - t^2 e_5 $ \\ \hline

$\mathbb{S}^{t+\frac{1}{4}}_{22}$&$\to$&$\mathbb{S}_{24}$ & 
\multicolumn{2}{l|}{$E_1^t=\left(\frac{4t}{4t^2-1}\right)^{\frac{1}{2}}e_1+\left(\frac{2}{4t^2-1}\right)e_2-\left(\frac{1}{4t^2-1}\right)e_4$} \\ 

\multicolumn{3}{|l}{ $E_2^t=\left(\frac{4t}{4t^2-1}\right)e_2$} &
\multicolumn{2}{l|}{$E_3^t=\left(\frac{4t}{4t^2-1}\right)^{\frac{3}{2}}e_3+\frac{-8t^2+2t}{(4t^2-1)^2}e_5$}\\

\multicolumn{3}{|l}{$ E_4^t=-\frac{(4t)^{\frac{3}{2}}}{(4t^2-1)^{\frac{5}{2}}}e_3+\left(\frac{4t}{4t^2-1}\right)e_4$} &
\multicolumn{2}{l|}{$E_5^t=\left(\frac{4t}{4t^2-1}\right)^2e_5$} \\ \hline

$\mathbb{S}^{-\frac{\alpha}{(\alpha-1)^2}}_{22}$&$\to$&$\mathbb{S}_{30}^{\alpha\neq 1}$ & 
$E_1^t=te_1$ & $E_2^t=e_2+\left(\frac{1}{\alpha-1}\right)e_4$ \\
\multicolumn{3}{|l}{ $E_3^t=te_3$} & $ E_4^t=t^2(\alpha-1)e_4$ & $E_5^t=t^2e_5$ \\ \hline


$\mathbb{S}_{30}^{t-1}$&$\to$&$\mathbb{S}_{34}$ & 
$E_1^t=e_1+t^{-1}e_4$ & $E_2^t=t^{-1}e_2$ \\ \multicolumn{3}{|l}{$E_3^t=t^{-1}e_3-(t-1){t^{-2}}e_5$} & $ E_4^t=e_4$ & $E_5^t=t^{-1}e_5$ \\ \hline

$\mathbb{S}_{30}^{t-1}$&$\to$&$\mathbb{S}_{35}$ & 
$E_1^t=e_1+\frac{t}{1+t^3}e_2-t^{-2}e_4$ & $E_2^t=\frac{t^2}{1+t^3}e_2+\frac{1}{t}e_4$ \\ 
\multicolumn{3}{|l}{$E_3^t=\frac{t^2}{1+t^3}e_3+\frac{t-2}{1+t^3}e_5$} & $ E_4^t=\frac{t-t^2}{1+t^3}e_3+e_4$ & $E_5^t=\frac{t^2}{1+t^3}e_5$ \\ \hline

$\mathbb{S}_{30}^{t-1}$&$\to$&$\mathbb{S}_{39}$ & 
$E_1^t=e_1+\frac{t^2}{t^5+1}e_2-\frac{1}{t^3}e_4$ & $E_2^t=\frac{t^4}{t^5+1}e_2+\frac{1}{t}e_4$ \\ \multicolumn{3}{|l}{$E_3^t=\frac{t^4}{t^5+1}e_3+\frac{t(t-2)}{t^5+1}e_5$} & $ E_4^t=\frac{t^3}{t^5+1}e_3+te_4$ & $E_5^t=\frac{t^4}{t^5+1}e_5$ \\ \hline


$\mathbb{S}_{41}^{-\frac{t^{-4}}{4}}$&$\to$&$\mathbb{S}_{42}$ & 
$E_1^t=t^{-2}e_1$ & $E_2^t=\frac{t^{-3}}{2}e_1+t^{-1}e_2$ \\ 
\multicolumn{3}{|l}{$E_3^t=t^{-3}e_3-\frac{t^{-5}}{2}e_5$} & $ E_4^t=t^{-4}e_4+\frac{t^{-6}}{2}e_5$ & $E_5^t=t^{-5}e_5$ \\ \hline

$\mathbb{S}_{45}$&$\to$&$\mathbb{S}_{47}$ & 
$E_1^t=\left(\frac{4}{4t^2-1}\right)^{\frac{3}{4}}te_1$ & $E_2^t=\frac{1}{2}\left(\frac{4}{4t^2-1}\right)^{\frac{3}{4}}e_1+\left(\frac{4}{4t^2-1}\right)^{\frac{1}{4}}e_2$\\ 
\multicolumn{3}{|l}{$E_3^t=\left(\frac{4}{4t^2-1}\right)te_3-\frac{1}{2}\left(\frac{4}{4t^2-1}\right)^{\frac{3}{2}}te_5$} 
 & $ E_4^t=\left(\frac{4}{4t^2-1}\right)^{\frac{5}{4}}te_4$ & $E_5^t=\left(\frac{4}{4t^2-1}\right)^{\frac{3}{2}}te_5$ \\ \hline

 
\end{longtable}

For the rest of degenerations, in  case of  $E_1^t,\dots, E_n^t$ is a {\it parametric basis} for ${\bf A}\to {\bf B},$ it will be denoted by
${\bf A}\xrightarrow{(E_1^t,\dots, E_n^t)} {\bf B}$. 


\begin{longtable}{|lcl|lcl|}
  
\hline
$\mathbb{S}_{01}   $ & $  \xrightarrow{ (t^2e_1,te_2,t^3e_3,t^4e_4,t^2e_5)}$&$ \mathbb{S}_{02}$ & 
$\mathbb{S}_{01} $ & $  \xrightarrow{ (t^{-2}e_1,t^{-1}e_2,t^{-3}e_3,t^{-4}e_4,t^{-3}e_5)} $ &   $ \mathbb{S}_{03}$ \\ \hline

$\mathbb{S}_{06} $ & $  \xrightarrow{ (t^{-1}e_1,t^{-\frac{1}{2}}e_2,t^{-\frac{3}{2}}e_3,t^{-2}e_4,t^{-2}e_5)} $ &   $ \mathbb{S}_{04}$ &
$\mathbb{S}_{06} $ & $  \xrightarrow{ (e_1,te_2,te_3,t^2e_4,e_5)} $ &   $ \mathbb{S}_{07}$ \\ \hline 

$\mathbb{S}_{21}^{1+t,t} $ & $  \xrightarrow{ (t^{-\frac{1}{2}}e_1,e_2,t^{-\frac{1}{2}}e_3,t^{-\frac{1}{2}}e_4,t^{-1}e_5)} $ &   $ \mathbb{S}_{08}$ &  $\mathbb{S}_{21}^{t,t} $ & $  \xrightarrow{ (t^{-\frac{1}{2}}e_1,e_2,t^{-\frac{1}{2}}e_3,t^{-\frac{1}{2}}e_4,t^{-1}e_5)} $ &   $ \mathbb{S}_{09}$ 
\\ \hline

$\mathbb{S}_{21}^{\frac{1}{t^2},\frac{\alpha}{t}} $ & $  \xrightarrow{ (te_1,e_2,te_3,te_4,e_5)} $ &   $ \mathbb{S}_{10}^{\alpha}$ & 
$\mathbb{S}_{10}^{t^{-\frac{1}{2}}} $ & $  \xrightarrow{ (t^{-\frac{1}{2}}e_1,e_2,t^{-\frac{1}{2}}e_3,t^{-\frac{1}{2}}e_4,t^{-1}e_5)} $ &   $ \mathbb{S}_{11}$\\ \hline 

$\mathbb{S}_{10}^{t^{-1}} $ & $  \xrightarrow{ (te_1,e_2,te_3,te_4,e_5)} $ &   $ \mathbb{S}_{12}$ &
$\mathbb{S}_{12} $ & $  \xrightarrow{ (te_1,e_2,te_3,te_4,e_5)} $ &   $ \mathbb{S}_{13}$ \\ \hline 

$\mathbb{S}_{12} $ & $  \xrightarrow{ (t^{-2}e_1,t^{-1}e_2,t^{-3}e_3,t^{-4}e_4,t^{-3}e_5)} $ &   $ \mathbb{S}_{14}$ & $\mathbb{S}_{11} $ & $  \xrightarrow{ (t^{-\frac{1}{2}}e_1,e_2,t^{-\frac{1}{2}}e_3,t^{-\frac{1}{2}}e_4,t^{-1}e_5)} $ &   $ \mathbb{S}_{15}$  \\ \hline

$\mathbb{S}^{\frac{1-t^2}{t^2},0}_{21}$&
$ \xrightarrow{
(t^{-1}e_1+t^{-1}e_2, t^{-2}e_2,  t^{-3}e_3+t^{-3}e_4,  t^{-5}e_4, t^{-4}e_5)}$&$\mathbb{S}_{16}$ & 

$\mathbb{S}_{16} $ & $  \xrightarrow{ (t^{-\frac{1}{2}}e_1,e_2,t^{-\frac{1}{2}}e_3,t^{-\frac{1}{2}}e_4,t^{-1}e_5)} $ &   $ \mathbb{S}_{17}$ 
\\ \hline

$\mathbb{S}_{16} $ & $  \xrightarrow{ (t^{-\frac{1}{2}}e_1,t^{-1}e_2,t^{-\frac{3}{2}}e_3,t^{-\frac{5}{2}}e_4,t^{-2}e_5)} $ & $ \mathbb{S}_{18}$ &
$\mathbb{S}_{18}$ & $\xrightarrow{ (t^{-\frac{1}{2}}e_1,e_2,t^{-\frac{1}{2}}e_3,t^{-\frac{1}{2}}e_4,t^{-1}e_5)} $ &   $ \mathbb{S}_{19}$ \\ \hline

$\mathbb{S}_{21}^{\frac{1}{t^2},t} $ & 
$  \xrightarrow{ 
(t^{-1}e_1, t^{-1}e_2,t^{-2}e_3,t^{-3}e_4,t^{-4}e_5)} $ &   $ \mathbb{S}_{20}$ 
&
$\mathbb{S}^{\alpha}_{22} $ & $  \xrightarrow{ (t^{-1}e_1,t^{-2}e_2,t^{-3}e_3,t^{-2}e_4,t^{-4}e_5)} $ &   $ \mathbb{S}_{23}^\alpha$ \\ \hline

$\mathbb{S}^{t^{-2}}_{22} $ & $  \xrightarrow{ (t^{-1}e_1,e_2,t^{-1}e_3,t^{-1}e_4,t^{-2}e_5)} $ &   $ \mathbb{S}_{25}$ &
$\mathbb{S}_{25} $ & $  \xrightarrow{ (t^{-1}e_1,t^{-2}e_2,t^{-3}e_3,t^{-2}e_4,t^{-4}e_5)} $ &   $ \mathbb{S}_{26}$ 

 \\ \hline

$\mathbb{S}_{24} $
& $  \xrightarrow{ (e_1,t^2e_2,t^2e_3,te_4,t^2e_5)} $ &   $ \mathbb{S}_{27}$ 
&

$\mathbb{S}_{25} $&
$  \xrightarrow{ (t^{-1}e_1,e_2,t^{-1}e_3,t^{-1}e_4,t^{-2}e_5)} $ &   $ \mathbb{S}_{28}$ \\ \hline

$\mathbb{S}_{28} $ & $  \xrightarrow{ (te_1,t^{-2}e_2,t^{-1}e_3,e_4,e_5)} $ &   $ \mathbb{S}_{29}$
&

$\mathbb{S}_{30}^{\alpha} $ & $  \xrightarrow{ (te_1,t^{-2}e_2,t^{-1}e_3,t^2e_4,e_5)} $ &   $ \mathbb{S}_{31}^{\alpha}$ \\ \hline

$\mathbb{S}_{30}^{t-1}$&
$ \xrightarrow{(e_1,  e_2+t^{-1}e_4, e_3,  e_4, e_5)}$&
$\mathbb{S}_{32}$   
&

$\mathbb{S}_{35} $ & $  \xrightarrow{ (t^{-1}e_1,t^{-2}e_2,t^{-3}e_3,t^{-2}e_4,t^{-4}e_5)} $ &   $ \mathbb{S}_{33}$\\ \hline

$\mathbb{S}_{35} $ & $  \xrightarrow{ (t^{-1}e_1,t^{-1}e_2,t^{-2}e_3,t^{-2}e_4,t^{-3}e_5)} $ &   $ \mathbb{S}_{36}$ 
&

$\mathbb{S}_{30}^{t^{-1}} $ & $  \xrightarrow{ (e_1,e_2,e_3,te_4,e_5)} $ &   $ \mathbb{S}_{37}$\\ \hline

$\mathbb{S}_{37} $ & $  \xrightarrow{ (te_1,t^{-2}e_2,t^{-1}e_3,t^2e_4,e_5)} $ &   $ \mathbb{S}_{38}$  &

$\mathbb{S}_{39} $ & $  \xrightarrow{ (e_1,te_2,te_3,te_4,te_5)} $ &   $ \mathbb{S}_{40}$ 
\\ \hline

$\mathbb{S}_{41}^{t^{-3}} $ & $  \xrightarrow{ (t^{-2}e_1,t^{-1}e_2,t^{-3}e_3,t^{-4}e_4,t^{-5}e_5)} $ &   $ \mathbb{S}_{43}$ &
$\mathbb{S}_{43} $ & $  \xrightarrow{ (t^{-2}e_1,t^{-1}e_2,t^{-3}e_3,t^{-4}e_4,t^{-5}e_5)} $ &   $ \mathbb{S}_{44}$
  \\ \hline

$\mathbb{S}_{41}^{t^4} $ & $  \xrightarrow{ (t^3e_1,te_2,t^4e_3,t^5e_4,t^6e_5)} $ &   $ \mathbb{S}_{45}$ &

$\mathbb{S}_{45} $ & $  \xrightarrow{ (t^{-3}e_1,t^{-1}e_2,t^{-4}e_3,t^{-5}e_4,t^{-6}e_5)} $ &   $ \mathbb{S}_{46}$ 
\\ \hline

$\mathbb{S}_{47}$&
$\xrightarrow{(te_1,  e_1+t^{-\frac{1}{2}}e_2,  t^{\frac{1}{2}}e_3,  e_4,  t^{-\frac{1}{2}}e_5)}$&$\mathbb{S}_{48}$ &

$\mathbb{S}_{48} $ & $  \xrightarrow{ (t^{-2}e_1,te_2,t^{-1}e_3,e_4,te_5)} $ &   $ \mathbb{S}_{49}$ 
 \\ \hline
\end{longtable} 

\end{proof}

\

\end{document}